\documentclass[12pt]{article}

\usepackage{epic,eepic,euscript,verbatim,amsmath,amssymb,
amsthm,amsfonts,afterpage,float,bm,stmaryrd,paralist,
mathrsfs,color,multirow,cases}
\usepackage{booktabs,enumitem}
\usepackage{csvsimple}
\usepackage{colordvi}
\usepackage{graphicx}
\usepackage{subfigure}
\usepackage{caption}
\usepackage{wrapfig}
\usepackage{algpseudocode}
\usepackage{authblk}
\usepackage{lipsum}
\usepackage[titletoc]{appendix}
\usepackage{setspace}
\usepackage[numbers,sort&compress]{natbib}
\usepackage{lineno}

\newtheorem{theorem}{Theorem}[section]
\newtheorem{lemma}{Lemma}[section]
\newtheorem{proposition}{Proposition}[section]

\newtheorem{definition}{Definition}[section]
\newtheorem{remark}{Remark}[section]

\numberwithin{equation}{section}
\numberwithin{figure}{section}

\def\XXint#1#2#3{{\setbox0=\hbox{$#1{#2#3}{\int}$}
\vcenter{\hbox{$#2#3$}}\kern-.51\wd0}}

\newcommand{\R}{\mathbb{R}}

\newcommand{\ve}{\varepsilon}
\newcommand{\cA}{\mathcal{A}}
\newcommand{\cB}{\mathcal{B}}
\newcommand{\cE}{\mathcal{E}}
\newcommand{\cL}{\mathcal{L}}
\newcommand{\cT}{\mathcal{T}}

\newcommand{\dist}{\operatorname{dist}}

\begin{document}
\title{Maximizing Reaction via Spatial Localization: 
A PDE-Constrained Variational Approach} 

\author{Mohammad Yasir Feroz Khan} 
\author{Bo Li}

\affil{Department of Mathematics\\
University of California San Diego\\
La Jolla, California, USA
\\
Email: m5khan@ucsd.edu and  bli@math.ucsd.edu
}

\date{ }

\maketitle

\begin{abstract}
We study a PDE-constrained variational problem for the optimal spatial distribution of a catalytic material reacting with a substrate in a biological or materials system.
With a given catalyst concentration, the substrate concentration is determined by 
a reaction-diffusion equation. The reaction functional for the catalyst concentrations is 
defined to be the total reaction output, which is the product of catalyst
and substrate concentrations. Both the linear and Michaelis--Menten reaction kinetics are 
considered. We construct reaction-maximizing sequences through 
the localization of catalysts
using barrier functions and show the non-existence of a global maximizer. 
We also derive the first and second variations of the reaction functional. By constructing a competing catalyst concentration, we show the non-existence of 
a local maximizer. Based on our analysis, we propose
a regularized variational model and prove the existence of its global maximizer.

\vspace{2 mm}


\vspace{2 mm}


\end{abstract}

\allowdisplaybreaks

\section{Introduction}
\label{s:Intro}

Chemical reactions are fundamental biological, chemical, and material processes,
where catalytic molecules or materials interact with
substrates that are targeted molecules or materials to result in reaction products. 
Reaction rates often measure the efficiency of such a process
\cite{Cornish-Bowden_2012,Lente_2015}. Spatial distributions of catalysts, however, can significantly 
affect such efficiency. For instance, experiments and theoretical biophysical 
studies have revealed that various types of spatial distributions of enzymes in cells can optimize
the reaction flux, the total amount of reaction product, under different biological constraints.  
Such optimal distributions occur in cell
signaling pathways that consist of sequential enzymatic reactions and that are crucial to biological functions
\cite{Lente_2015,Marcotte_Rev2012,Srere_EnzymeComplexes_Rev1987,Chenoweth_JACS2018}. 
Patterns of such distributions include localization and co-localization, 
clustering and co-clustering, sequestration, 
and compartmentalization \cite{Gerland_CommunPhys2022,Gerland_JCP2013,Gerland_PRL2013,Gerland_NatPhys2022,Wingreen_NatBiotech2014,Hinzpeter_BiophysJ2017,tenWolde_PoLS2007,An_Biochem2017,Gerland_Interface2019}.
Intuitively, the reaction flux, which quantifies the reaction efficiency, should increase
if enzyme molecules are localized in the spatial regions where the substrate molecules 
are mostly concentrated. Here, we 
carry out mathematical analysis of a variational model for maximizing the reaction flux, 
focusing on understanding the mechanisms of such localization.

To be more precise, we
study the variational problem of maximizing the reaction functional 
(also termed the total reaction flux)
\begin{equation}
\label{Ju}
J[u] = \int_\Omega u f(\rho_u) \, dx 
\end{equation}
over the class of admissible catalyst (e.g., enzyme molecule) concentrations 
\begin{align}
\label{cE}
& \cE = \left\{ u \in L^2(\Omega): u \ge 0 \mbox{ a.e.\ in } \Omega  \mbox{ and }
 \int_\Omega u \, dx = M_{\rm e} \right\},
\end{align}
where $\rho = \rho_u \in H^1(\Omega)$ 
is the substrate concentration
 defined to be the unique, nonnegative solution
to the boundary-value problem
\begin{equation}
    \label{BVP4rho}
\left\{
\begin{aligned}
    & - \Delta \rho+ \sigma \rho + u f(\rho) = 0 & &\quad  \mbox{in } \Omega, & \\
    & \rho = g &  &\quad \mbox{on } \partial \Omega. & 
\end{aligned}
\right.
\end{equation}
Here, $\Omega \subset \R^3$ is a bounded and smooth domain, 
$M_{\rm e}>0$ and $\sigma > 0$ are given constants, $g \in H^2(\Omega)$ is given and satisfies 
$g \ge 0$ in $\Omega$, and the function $f$ is defined for 
two different cases by 
\begin{equation}
\label{f}
f(s) = \left\{
\begin{aligned}
    & s & &  \quad \mbox{for the linear case}, \\
    &  \frac{s}{s+1} & & \quad \mbox{for the nonlinear case}.
\end{aligned}
\right.
\end{equation}
The nonlinear form is the non-dimensionalized Michaelis--Menten kinetics
\cite{Cornish-Bowden_2012}. 

Our construction of the model, which is in a dimensionless form, follows from 
\cite{Gerland_PRL2013,Gerland_JCP2013} (cf.\ also
\cite{Gerland_CommunPhys2022,Gerland_NatPhys2022,Rios_JCP2025,Plasson_PRE2025}). 
Both the reaction coefficient in 
front of the reaction product $u f(\rho_u)$ and the diffusion coefficient in front of the 
term $\Delta \rho$ in \eqref{BVP4rho} are  rescaled to be the unity.
The $\sigma$ term in \eqref{BVP4rho} describes the loss of substrate via a process other than the reaction.
As in \cite{Gerland_PRL2013,Gerland_JCP2013}, we consider the steady-state substrate
concentration $\rho_u$; cf.\ \eqref{BVP4rho}. It is
 determined by the admissible enzyme concentration $u$, together with the boundary value $g$ 
that determines the spatial regions where the concentration $\rho_u$ is large.
Focusing on the localization, we use the Dirichlet instead of Neumann boundary conditions. 
While other effects, such as transport via fluid 
flow, electrostatic interaction, 
and a multi-stage reaction process, are excluded, our analysis will show that 
the PDE-constrained variational model has the essential components
of reaction to facilitate the catalyst spatial localization that can be experimentally implemented
(cf. e.g., \cite{Chenoweth_JACS2018}).

Our main results include the non-existence
of local maximizers of the flux functional $J$ over the set $\cE$ of admissible enzyme concentrations, 
and the existence is recovered through a regularization. More details
of our results and the key ideas of their proofs
are outlined below.

First, we construct a maximizing sequence of catalyst concentrations $u_\ve \in \cE$ 
for the reaction functional $J$ defined in \eqref{Ju}: 
$
\lim_{\ve \to 0^+} J[u_\ve] = \sup_{u \in \cE} J[u],
$
and show that the functional $J$ does not admit any global maximizer; cf.\
Theorem~\ref{t:supJ} and Theorem~\ref{t:uve}. Localization is the key
in such constructions: Each $u_\ve$ is a scaled characteristic function supported near a boundary point where $g$ attains its maximum. As $\ve \searrow 0$, the volume of the support of 
 $u_\ve$ converges to $0$, while its value tends to infinity. For a smooth and bounded domain $\Omega \subset \R^3$, the 
 construction of $u_\ve$ is done by using barrier functions and different length scales given by powers of $\ve.$ However, the essence behind such localization can be seen
 from a one-dimensional setting for which we can do explicit calculations. 

 Second, we prove that in fact the functional $J: \cE \to \R$ does not 
 permit a local maximizer with respect to the $L^2(\Omega)$-norm. 
 To show this, we calculate the first and second variations of the reaction functional
$J;$ cf.\ Theorem~\ref{t:FirstVariations} and Theorem~\ref{t:SecondVariations}.
A consequence of these calculations is that, if $u \in\cE$ is a critical point, then it 
satisfies 

\[
f(\rho_u) ( 1 - \psi_u)=c_u \quad \mbox{ a.e.\ in }\,  \{u>0\} \qquad \mbox{and} \qquad 
| \{ u = 0 \} | > 0, 
\]
where $c_u\ge0$ is a constant and
$\psi_u\in H_0^1(\Omega)\cap H^2(\Omega)$ solves the
corresponding adjoint equation.
Here and below, 
$|A|$ denotes the Lebesgue measure of a Lebesgue-measurable
set $A$.
If $u$ is also a local maximizer, then we have
the first variation $J'_+[u][v] \le 0$ for any one-sided perturbation $v$ 
defined by $u + t v \in \cE$ for all sufficiently small $t>0.$ This allows us to construct suitable
perturbations by moving some mass from $\{ u>0 \}$ into $\{ u = 0 \}$ to obtain
\[
f(\rho_u) ( 1 - \psi_u)\le c_u \quad \mbox{ a.e.\ in }\,  \{u= 0\}. 
\]
By the previous equality, this inequality then holds true on the entire $\Omega,$ 
which, together with the boundary condition and other properties of $\rho_u$, will lead to a contradiction.

Third, we propose a regularized variational model with the new reaction functional
\begin{equation}
\label{Jveu}
\hat{J}[u]= \int_\Omega \left[ u f(\rho_u) - \frac{1}{\beta} u \log u - \frac{\ve}{2} |\nabla u|^2 \right] dx,
\end{equation}
where $\beta > 0$ is a constant inversely proportional to temperature, $\ve > 0$ is a small and fixed parameter, 
and the function $\rho = \rho_u$ is the unique solution to the boundary-value problem \eqref{BVP4rho}. 
The second term of $\hat{J}$ describes the entropic effect of the catalyst distributions, while the last 
term, the gradient-square term, is a mathematical regularization. 
Using the direct method in the calculus of variations, we prove the existence of a global maximizer among all  $u \in\cE\cap H^1(\Omega)$. We will address further questions about the new model, such as the asymptotic behavior of the functional and its maximizers as $\ve \searrow 0$, as well as the gradient dynamics, in future work.

From the proof of the non-attainment of a global 
maximizer (cf.\ section~\ref{s:Local}), we see that the localization, 
together with the mass conservation, of enzyme (concentration) $u$ 
is the key for the reaction flux $J[u]$ to be as large as possible. Therefore, 
the model captures the essential physical effect. 
The lack of control of the gradient of enzyme concentration $u$ 
in the flux functional $J: \cE \to \R $, which is bounded,
leads to the unlimited localization, and hence
 the non-existence of a maximizer. With the gradient-square term included, 
 the regularized model \eqref{Jveu}, where the coefficient $\ve > 0$ 
 sets the length scale that limits
 the localization, admits a global maximizer.  
 In the gradient-ascent
 dynamics, this gradient-square term corresponds to the diffusion of enzyme molecules. 
In addition to diffusion, another physical effect that can prevent the non-attainment of a global maximizer 
is the entropic effect, as more crowded and localized molecules can decrease the conformational entropy, 
increasing the free energy and decreasing the
thermodynamic stability. In the regularized model, such effect is 
captured by the term $- u \log u$; cf.\ \eqref{Jveu}. 
We note that the molecular size effect arising from the minimum size of an underlying
enzyme molecule is not included in our regularized model. 
Such an effect may prevent the support of 
enzyme concentration $u$ from being infinitesimally small and hence unlimited localization 
due to the constraint of mass conservation. 
This is reminiscent of the ionic size effect in the 
Poisson--Boltzmann (PB) theory for electrostatic interactions. 
The electrostatic potential predicted by the classical PB theory can be unreasonably 
large near a charged wall, in contrast to experimental observations. With the ionic size effect 
incorporated into the continuum PB model, the potential is bounded; cf.\ 
\cite{BAO_PRL97,IglicIglic94,Li_Nonlinearity09,LiLiuXuZhou_Nonlinearity2013,ZhouWangLi_PRE11}.

The rest of the paper is organized as follows. In section~\ref{s:Model}, we study the 
basic properties of the substrate concentration $\rho = \rho_u$
as the solution to the boundary-value problem \eqref{BVP4rho} and its continuous dependence 
on the catalyst concentration $u.$ 
In section~\ref{s:Local}, we construct a maximizing sequence $u_\ve \in \cE$ $(0 < \ve\ll 1)$ for the reaction functional
$J: \cE \to \R$.
We also show the
non-existence of a global maximizer. 
In section~\ref{s:FirstVar}, we calculate the first and second variations
of the reaction functional $J: \cE \to \R$.  
In section~\ref{s:CriticalPoints}, we show the non-existence of a local maximizer of the reaction functional
$J: \cE\to \R.$ Finally, in section~\ref{s:Regularization}, we show that the proposed regularized
variational model admits a global maximizer.

\section{Some Preliminary Results}
\label{s:Model}
Unless otherwise stated, we assume the following throughout: 
\begin{compactenum}
    \item[(A1)]
    The set $\Omega \subset \R^3$ is a bounded domain (i.e., open and connected)
    with a $C^2$-boundary $\partial \Omega;$
\item[(A2)]
Both $M_{\rm e} > 0$ and $\sigma > 0$ are given constants;  
\item[(A3)]
The function $g\in  H^2(\Omega)\hookrightarrow C(\overline{\Omega})$ is given. 
It satisfies $g \ge 0$ in  $\overline{\Omega}$ and there exists 
$y_0 \in \partial \Omega$ such that 
\begin{equation}
\label{gmax}
g_{\rm max} := g(y_0) = \max_{y \in \partial \Omega } g(y)
= \| g \|_{L^\infty(\partial \Omega)}> 0.
\end{equation}

\end{compactenum}
For any $p \in [1, \infty]$ and any $h\in H^1(\Omega)$, we denote, respectively, 
\begin{align*}
&L^p_+ (\Omega) = \{ u \in L^p(\Omega): u \ge 0 \mbox{ a.e.\ in } \Omega \}, 
\\
&
H^1_h(\Omega) = \{ v  \in H^1(\Omega): v= h \mbox{ on } \partial \Omega \}, 
\end{align*}
where the boundary value is understood in the trace sense. 
(We use the standard notations of Lebesgue and Sobolev 
spaces \cite{Adams75,folland,GilbargTrudinger98}.)
For any Lebesgue-measurable set $A \subseteq \R^3$, 
we denote by $|A|$ the Lebesgue measure of $A.$

Let $u \in L^2_+(\Omega)$. We define $I_u: H^1(\Omega) \to [0, \infty] $ by 
\begin{equation}
    \label{Iu}
I_u [ \rho ] = \int_\Omega \left[ \frac12 | \nabla \rho |^2 + \frac{\sigma}{2} \rho^2 + u F(\rho)
\right] dx, 
\end{equation}
where 
\begin{equation}
\label{F}
F( s ) = \left\{ 
\begin{aligned}
   & \frac12 s^2 \quad & & \mbox{for the linear case},\\
   &  s - \log (1 + s) \quad & & \mbox{for the nonlinear case}. 
\end{aligned}
\right.
\end{equation}
Here and below $\log $ denotes the natural logarithm with $\log s = -\infty$ if $s \le 0$ and
$s \log s = 0$ if $s = 0.$
Note that the function $F$ satisfies the following: 
\begin{compactenum}
    \item[(F1)] $F$ is smooth, nonnegative, and convex in its domain $(-\infty, \infty)$ or 
    $(-1, \infty)$ for the linear
    or nonlinear case, respectively; 
    \item[(F2)] For the nonlinear case: $F:(-1, \infty) \to \R$ 
    decreases in $(-1, 0)$ and increases in $(0, \infty).$ 
    In particular, $F$ attains its minimum $0$ uniquely at $s =0.$ Moreover, 
   $F(s) < s $ for any $s > 0$ and $F(s)/s \to 1 $ as $s \to +\infty.$
   \item[(F3)] With $f$ defined in \eqref{f}, 

\begin{equation}
\label{Fpf}
F'(s) = f(s) = 
\left\{ 
\begin{aligned}
   & s \quad & & \mbox{if } -\infty < s < \infty, & & \quad \mbox{for the linear case}, \\
   &  \frac{s}{s+1} & & \mbox{if }  s >-1,  & & \quad \mbox{for the nonlinear case}. 
\end{aligned}
\right.
\end{equation}
\end{compactenum}

\begin{theorem}
    \label{t:rhou}
   For $u \in L^2_+(\Omega)$, the functional $I_u: H_g^1(\Omega) \to [0, \infty]$ admits a 
    unique minimizer $\rho_u\in H_g^1(\Omega):$ 
    \begin{equation}
    \label{minIu}
    I_u [ \rho_u ] = \min_{\xi \in H_g^1(\Omega)} I_u [\xi].
    \end{equation} 
The unique minimizer $\rho_u$ is the unique nonnegative weak solution to the boundary-value problem
\eqref{BVP4rho}, 
defined by $\rho_u \in H_g^1(\Omega)$, $\rho_u \ge 0$ a.e.\  in $\Omega,$ and 
\begin{equation}
\label{LinftyH01}
\int_\Omega\left[ \nabla \rho_u \cdot \nabla \phi + \sigma \rho_u \phi
+ u f(\rho_u) \phi \right] dx = 0 \qquad \forall \phi \in H_0^1(\Omega). 
\end{equation}
Moreover, $\rho_u \in H^2(\Omega)\hookrightarrow C(\overline{\Omega})$ and    
$0 \le \rho_u < g_{\rm max}$ in $\Omega.$ 
If $| \{ u > 0 \}| > 0$, then
$| \{ \rho_u > 0 \} \cap \{ u > 0 \} | > 0.$ 
\end{theorem}

\begin{proof}
Since $u \ge 0$ a.e.\ in $\Omega$ and $F \ge 0$, we have 
\begin{equation}
\label{Irhomin}
I_u [\rho] \ge \frac{1}{2}\min \{ 1, \sigma \} \| \rho \|_{H^1(\Omega)}^2 
\qquad \forall \rho \in H_g^1(\Omega). 
\end{equation}
Thus, $\alpha:=\inf_{\rho \in H_g^1(\Omega) }I_u[\rho] \ge 0.$
Moreover, $\alpha \le I_u[g] < \infty.$ If $\rho \in H_g^1(\Omega),$ then 
$\rho^+:=\max \{\rho, 0\}\in H_g^1(\Omega)$. Moreover, $I_u[\rho^+] \le I_u[\rho]$, since
$(\rho^+)^2 \le \rho^2, $  $ |\nabla \rho^+| \le |\nabla \rho|, $ and 
$u F(\rho^+) \le u F(\rho)$ a.e.\ in $\Omega$. 
Thus, minimizing $I_u$ over $H_g^1(\Omega)$ is the same as minimizing 
$I_u$ over the class of nonnegative functions in $H_g^1(\Omega).$ Note that $g$ is such a function.

Let $\rho_k \in H_g^1(\Omega)$ $(k = 1, 2, \dots) $ be such that $I_u[\rho_k]\to \alpha.$ We may assume $\rho_k \ge 0$ 
a.e.\  in $\Omega$ as $I_u[\rho_k^+] \le I_u [\rho_k]$ for each $k$. 
It follows from \eqref{Irhomin} that  $\{ \rho_k \}$
is bounded in $H^1(\Omega)$. Thus, it has a subsequence, not relabeled, such that 
$\rho_k\to \rho_u$ weakly in $H^1(\Omega),$ strongly in $L^2(\Omega)$, and a.e.\ in $\Omega$ for some $\rho_u \in H_g^1(\Omega)$
with $\rho_u \ge 0$ a.e.\ in $\Omega.$
Now, by the weak lower semi-continuity of nonnegative quadratic functionals and Fatou's Lemma, we have
\[
\alpha = \liminf_{k\to \infty} I_u[\rho_k] \ge I_u[\rho_u] \ge \alpha, 
\]
leading to \eqref{minIu}. Since
$F $ is convex in its domain, 
$I_u$ is strictly convex on
the subset of $H_g^1(\Omega)$ on which $I_u$ is finite. Thus, the minimizer of $I_u$ is unique.

Let $\hat{\rho}_u = \min \{ \rho_u, g_{\rm max} \}.$ 
Then, $\hat{\rho}_u\in H_g^1(\Omega)$, $\hat{\rho}_u  \ge 0$ a.e.\ in $\Omega$, and 
$I_u[\hat{\rho}_u ] \le I_u[\rho_u].$ Hence, by the uniqueness of a minimizer of $I_u$ over 
$H_g^1(\Omega)$ and the fact that $\rho_u \ge 0$ a.e.\ in $\Omega,$ we have 
$0 \le \rho_u = \hat{\rho}_u  \le  g_{\rm max}$ a.e.\ in $\Omega$
and $\rho_u\in L^\infty(\Omega).$

Let $\phi \in H_0^1(\Omega)\cap L^\infty(\Omega)$. Since $\rho_u \ge 0$ a.e.\ in $\Omega$,
there exists $t_\phi >0$ such that $\rho_u + t\phi \ge -1/2 $
a.e.\ in $\Omega$ if $|t| \le t_\phi$. 
Since $\rho_u + t \phi \in H_g^1(\Omega)$ and hence $I_u[\rho_u] \le I_u[\rho_u + t \phi]$
$(|t|\le t_\phi)$, we have 
$({d}/{dt})|_{t=0} I_u [ \rho_u + t \phi ] = 0.$
For the linear case, we have 
\[
F(\rho_u + t \phi) - F (\rho_u) = t \rho_u \phi + \frac12 \phi^2 t^2 \qquad \mbox{if } |t| \le t_\phi.
\]
For the nonlinear case, we have 
\[
|F(\rho_u + t\phi) - F(\rho_u) | \le |t| |\phi| \qquad \mbox{if } |t| \le t_\phi,
\]
since by \eqref{Fpf} 
$|F'(s)| = |f(s)|\le 1$ for $ s \ge -1/2.$ Now, by routine calculations
using the Lebesgue Dominated Convergence Theorem, we obtain \eqref{LinftyH01} with $\phi
\in H_0^1(\Omega)\cap L^\infty(\Omega)$. Since $\rho_u \in L^\infty(\Omega)$ as shown above
and hence $u f(\rho_u) \in L^2(\Omega)$, and $H_0^1(\Omega)\cap L^\infty(\Omega)$
is dense in $H_0^1(\Omega)$, we obtain \eqref{LinftyH01} with $\phi \in H_0^1(\Omega).$

We now show that the nonnegative function $\rho \in H_g^1(\Omega)$ that satisfies \eqref{LinftyH01} 
(with $\rho$ replacing $\rho_u)$ is unique. 
Suppose $\tilde{\rho}_u\in H_g^1(\Omega)$ and $\tilde{\rho}_u \ge 0$ a.e.\ 
in $\Omega$ also satisfies
\eqref{LinftyH01}. Let $\eta_u = \rho_u- \tilde{\rho}_u\in H_0^1(\Omega).$ Then, for any $\phi \in H_0^1(\Omega)$, we have 
\begin{equation*}
\int_\Omega \left[\nabla \eta_u \cdot \nabla \phi + \sigma \eta_u \phi + u ( f(\rho_u) - f(\tilde{\rho}_u)) \phi \right] dx = 0. 
\end{equation*}
Setting $\phi = \eta_u$, we have
\begin{equation*}
\int_\Omega \left[ |\nabla \eta_u|^2 + \sigma \eta_u^2 + u (f (\rho_u) - f(\tilde{\rho}_u)) \eta_u \right] \, dx=0.
\end{equation*}
Since $u (f (\rho_u) - f(\tilde{\rho}_u)) \eta_u \ge 0$ for both the linear and nonlinear cases, each term in the integrand above is thus nonnegative, and the integral vanishes. Consequently, $\eta_u = 0$ in $H_0^1(\Omega)$ and $\tilde{\rho}_u = \rho_u$ in $H_g^1(\Omega). $

Finally, since $0\le f(\rho_u) \le f( \| \rho_u \|_{L^\infty(\Omega)})= f(g_{\rm max})$ in $\Omega$ and 
$\rho_u$ is a weak solution to \eqref{BVP4rho} (cf.\ \eqref{LinftyH01}), 
we have $\rho_u \in H^2(\Omega) $ by the regularity result (cf.\ Theorem~8.12 in \cite{GilbargTrudinger98}).
As $0 \le \rho_u \le g_{\rm max}$ a.e. in $\Omega$ by above, and $H^2(\Omega) \hookrightarrow C(\overline{\Omega})$, we have $0 \le \rho_u \le g_{\rm max}$ in $\Omega$. Now, suppose $\rho_u(x_0)=g_{\rm max}$ for some $x_0 \in \Omega$. Then, we can pick a ball $B \subset \Omega$ containing $x_0$ such that 
$
\sup_B \rho_u=\sup_\Omega \rho_u=g_{\rm max} > 0. 
$
This implies that  $\rho_u$ is constant in $\Omega$ by the Strong Maximum Principle 
applied to $\Delta - \sigma$ (cf.\ Theorem~8.19 in \cite{GilbargTrudinger98}). 
By continuity, $\rho_u \equiv g_{\rm max}$ in $\Omega$. But then \eqref{BVP4rho} becomes
$\sigma g_{\rm max}+uf(g_{\rm max})=0 $ in $\Omega$, 
a clear contradiction since the left-hand side of the equation is strictly positive. Thus, $0 \le \rho_u < g_{\rm max}$ in $\Omega$.

Suppose for $u \in L^2_+(\Omega)$ and $u \ne 0 $ in $L^2(\Omega)$, we have $| \{ \rho_u > 0 \} \cap \{ u > 0 \} | = 0. $ Then, $\rho_u = 0$ a.e.\ 
in $\{ u > 0 \}.$ Hence, it follows from \eqref{BVP4rho} with $\rho_u $ replacing 
$\rho$ that 
\begin{equation*}
    \left\{
    \begin{aligned}
        & - \Delta \rho_u + \sigma \rho_u  = 0 \quad & &\mbox{in } \Omega, \\
        & \rho_u = g \quad & & \mbox{on } \partial \Omega. 
    \end{aligned}
    \right. 
\end{equation*}
Note that $\rho_u \ge 0$ a.e.\ in $\Omega$. Since $\rho_u = 0$ a.e.\ in $\{ u > 0 \}$ and 
$| \{ u > 0 \} | > 0$, $\rho_u$ attains its minimum value in $\Omega$.  The Strong Maximum Principle 
(cf.\ Theorem~3.5 of \cite{GilbargTrudinger98})  implies that $\rho_u$ is a constant in $\Omega$. 
This constant is $0$, since $\rho_u = 0$ a.e.\ in $\{ u > 0 \}.$ This contradicts the 
fact that $\rho_u = g$ on $\partial \Omega$ as $\| g \|_{L^\infty(\partial \Omega)} = g_{\rm max} > 0.$
\end{proof}

We denote 
\begin{equation} 
\label{Tgu}
\cT_g(u) = \rho_u \qquad \forall u \in L^2_+(\Omega),
\end{equation} 
where $\rho_u$ is defined as in the above theorem. 
The theorem implies that the nonlinear operator 
$\cT_g$ maps $L^2_+(\Omega)$ to $H_g^1(\Omega)\cap H^2(\Omega). $
We define the reaction functional $J: L^2_+(\Omega) \to [0, \infty]$ by
\eqref{Ju} with $\rho_u = \cT_g(u)$.
The following lemma indicates that the operator $\cT_g: L^2_+(\Omega) \to H_g^1(\Omega)$
is continuous and that $J[u]$ for any $u \in \cE$ can be approximated by 
$J[u_k]$ with $u_k \in \cE \cap L^\infty(\Omega)$ $(k=1, 2, \dots)$.

\begin{lemma}
    \label{l:rhouk}
    Let $u \in L^2_+(\Omega)$ and $\rho_u = \cT_g(u)$. 
    \begin{compactenum}
        \item[{\rm (1)}]
If $u_k \in L^2_+(\Omega)$ $(k = 1, 2, \dots)$ and 
$u_k \to u $ in $L^2(\Omega)$, then $\rho_{u_k}:=\cT_g(u_k) \to \rho_u$ in $H^1(\Omega)$ and 
$J[u_k]\to J[u]$ as $k \to \infty.$
\item[{\rm (2)}]
If $u_k = \min\{ u, k \}\in L^\infty(\Omega)$ $(k = 1, 2, \dots)$, then $0 \le u_k \nearrow u$ 
a.e.\  in $\Omega$ and $u_k \to u$ in $L^2(\Omega)$, 
$\rho_{u_k} \to \rho_u$ in $H^1(\Omega)$, and $J[u_k] \to J[u]$ as $k \to \infty.$ 
        \item[{\rm (3)}]
If in addition $u \in \cE, $ then there exist $u_k \in \cE\cap L^\infty(\Omega)$ $(k=1,2,\dots)$ such that 
$u_k \to u$ in $L^2(\Omega)$ and  a.e.\  in $\Omega,$
$\rho_{u_k} \to \rho_u$ in $H^1(\Omega)$, and $J[u_k] \to J[u]$ as $k \to \infty.$ 
    \end{compactenum}
\end{lemma}

\begin{proof}
    (1) For each $k \ge 1$, we have 
by \eqref{LinftyH01} with $u_k$ replacing $u$ that 
\begin{equation}
\label{rhouk}
\int_\Omega\left[ \nabla \rho_{u_k} \cdot \nabla \phi + \sigma \rho_{u_k} \phi
+ u_k f(\rho_{u_k}) \phi \right] dx = 0 \qquad \forall \phi \in H_0^1(\Omega). 
\end{equation}
Setting $\phi = \rho_{u_k} - g\in H_0^1(\Omega)$ and noting 
that $u_k f(\rho_{u_k}) \rho_{u_k} \ge 0$ a.e.\  in $\Omega$, we obtain 
\[
\min \{ 1, \sigma \} \| \rho_{u_k} \|^2_{H^1(\Omega)} 
\le \int_\Omega \left[ \nabla \rho_{u_k} \cdot \nabla g + \sigma \rho_{u_k} g
+ u_k f(\rho_{u_k} ) g  \right] dx.
\]
Since $ \sup_{k \ge 1} \|u_k \|_{L^2(\Omega)} < \infty$ and $0 \le f(\rho_{u_k}) \le  \rho_{u_k}$, we further obtain
$\sup_{k \ge 1} \| \rho_{u_k} \|_{H^1(\Omega)}< \infty$. Thus, for any given subsequence
of $\{ u_k \}$, there exists a further subsequence, not relabeled, such that $\rho_{u_k} \to \rho$ weakly in $H^1(\Omega)$, 
strongly in $L^2(\Omega),$ and a.e.\ in $\Omega$ for some $\rho \in H_g^1(\Omega).$
Sending $k \to \infty$ in \eqref{rhouk} with $\phi \in H_0^1(\Omega)\cap L^\infty(\Omega)$, 
we infer that $\rho$ 
satisfies \eqref{LinftyH01} with $\rho$ replacing $\rho_u$ and $\phi \in H_0^1(\Omega)\cap L^\infty(\Omega).$
Since $0 \le \rho_{u_k} \le g_{\rm max}$ in $\Omega$ for each $k \ge 1,$ we have 
$0\le \rho \le g_{\rm max}.$
Hence, \eqref{LinftyH01} holds true with $\rho$ replacing $\rho_u$
and for all $\phi \in H_0^1(\Omega).$ The uniqueness of a nonnegative weak solution 
defined by \eqref{LinftyH01} now implies that $\rho = \rho_u. $

Denoting $\eta_k = \rho_{u_k} - \rho_u \in H_0^1(\Omega)$,  
setting $\phi = \rho_{u_k} - \rho_u\in H_0^1(\Omega)$ in \eqref{LinftyH01} and \eqref{rhouk}, and subtracting
one from the other, we get
\[
\int_\Omega \left[ |\nabla \eta_k|^2 + \sigma \eta^2_k + u_k \left(  f(\rho_{u_k}) -  f(\rho_u) \right) \eta_k
\right]dx = -\int_\Omega (u_k - u) f(\rho_u) \eta_k\, dx .
\]
Note that 
\[
u_k \left(  f(\rho_{u_k}) -  f(\rho_u) \right) \eta_k =
\left\{
\begin{aligned}
    & 
u_k \eta_k^2\ge 0 \quad & \mbox{a.e.\ in }\,  \Omega,  & \quad \mbox{for the linear case,}
\\
& \frac{ u_k \eta_k^2 }{(\rho_{u_k}+1)(\rho_u + 1)}
\ge 0 \quad  & \mbox{a.e.\ in }\, \Omega, &  \quad \mbox{for the nonlinear case.}
\end{aligned}
\right.
\]
Since $u_k \to u$ and $\eta_k \to 0$ in $L^2(\Omega)$, and $0 \le f(\rho_u) \le f( g_{\rm max})$
on $\Omega$, we see in both cases that $\eta_k \to 0$, i.e., $\rho_{u_k} \to \rho_u$ in $H^1(\Omega)$. 
This is true for a subsequence taken from any subsequence of $\{ u_k \}$. We 
thus conclude that the original
sequence $\{ \rho_{u_k} \}$ converges to $\rho_u$ in $H^1(\Omega).$
Finally, 
\begin{align*}
 \left|J[u_k]-J[u]\right|&=\left|\int_{\Omega} [u_k f(\rho_{u_k}) -uf(\rho_u) ]\, dx\right|\\
 &\le \int_{\Omega} |u_k| |f(\rho_{u_k}) -f(\rho_u) |\, dx +  \int_{\Omega} |u_k-u| |f(\rho_u) |\, dx \\
 & \le \|u_k\|_{L^2(\Omega)}\|\rho_{u_k}-\rho_u\|_{L^2(\Omega)} + \|u_k-u\|_{L^2(\Omega)}\|f(\rho_u)\|_{L^2(\Omega)} \\
 & \to 0
\end{align*}

(2) It follows from the definition of $u_k$ $(k=1,2,\dots) $ that  $0 \le u_k \nearrow u$ 
a.e.\  in $\Omega$ and $u_k \to u$ in $L^2(\Omega)$. The rest follows from Part (1).  

    (3) For each $ k \ge 1$, we define $u_k = \min \{ u, k \} + \delta_k,$ where
    \[
    \delta_k =  \frac{1}{|\Omega|}\int_{ \{ u > k \} } ( u - k)\,  dx \searrow 0 \qquad \mbox{as } k \to \infty. 
    \]
Clearly, $u_k \in \cE\cap L^\infty(\Omega)$ and $u_k \to u$ in $L^2(\Omega)$ and also 
a.e.\ in $\Omega$ as $k \to \infty.$ The rest follows from Part (1). 
\end{proof}

The next result concerns the normal derivative for a Sobolev function.

\begin{lemma}
\label{l:wnu}
Suppose $w \in H^2(\Omega)\cap H_0^1(\Omega)$ and $w \ge 0 $ a.e.\ in $\Omega$. Then $\partial_\nu w \le 0$ 
a.e.\ on $ \partial \Omega,$
where $\nu$ denotes the outward unit normal on $\partial \Omega$.
\end{lemma}

\begin{proof}
Let $\hat{\nu}:=-\nu$ denote the inward unit normal on $\partial \Omega.$
For $x \in \partial \Omega$ and $s \ge 0$ small enough, we define 
\[
F(x, s) = x + s \hat{\nu}(x) = x - s \nu(x). 
\]
Since $\partial \Omega$ is $C^2$, $\nu(x)$ is $C^1$. By the inverse function theorem and compactness of $\partial \Omega$, there exists $\ve>0$ such that
the mapping 
$F:\partial \Omega \times (0,\ve)\to \Omega_\ve$ 
is a $C^1$-diffeomorphism onto $
\Omega_\ve:=\{ y\in \Omega:\dist(y,\partial \Omega)<\ve\}.$ Since every point $y\in\Omega_\ve$ has a unique representation $y=F(x,s)$ with $x\in\partial\Omega$ and $0<s<\ve$, we extend $\hat{\nu}$ to $\Omega_\ve$ by keeping it constant along inward normal rays; namely, for $y=F(x,s)$, we set $\hat{\nu}(y):=\hat{\nu}(x).$ With this convention, $\partial_{s} F(x,s)=\hat{\nu} (x)=\hat{\nu}(F(x,s)).$ We also set $ \Sigma_s = F (\partial \Omega \times \{ s \})$ if  $0<s<\ve$. 
 
For any $\phi \in C(\overline{\Omega})$ 
and $0 < t < \ve$, we define 
\[
T_t\phi(x)=\frac{\phi(F(x,t))-\phi(F(x,0))}{t} =\frac{\phi(F(x,t))-\phi(x)}{t}\qquad \forall x\in\partial \Omega.
\]
Assume $\phi\in C^\infty(\overline{\Omega}).$ 
For $x\in\partial\Omega$ and  $0<t<\ve$, we have by the fundamental theorem of calculus that
\[
T_t \phi(x)
=\frac{1}{t}\int_0^t \frac{d}{ds}\bigl[\phi (F(x,s))\bigr]\,ds
=\frac{1}{t}\int_0^t \partial_{\hat{\nu}}\phi(F(x,s))\,ds,
\]
where $\partial_{\hat{\nu}}\phi (F(x,s)):=\nabla \phi (F(x,s))\cdot \hat{\nu}(F(x,s))$. 
Hence, by the Cauchy--Schwarz inequality,
\[
|T_t \phi (x)|^2
\le \frac1t\int_0^t |\partial_{\hat{\nu}} \phi (F(x,s))|^2\,ds.
\]
Integrating over $\partial\Omega$ and applying Fubini's theorem, we obtain
\[
\|T_t \phi \|_{L^2(\partial\Omega)}^2
\le \frac1t\int_0^t \left(\int_{\partial\Omega}
|\partial_{\hat{\nu}}\phi (F(x,s))|^2\,dS_x\right)\,ds.
\]
Moreover, the surface Jacobians $J_s$ of the maps $F(\cdot,s):\partial \Omega \to \Sigma_s$
depend continuously on $(x,s)\in \partial\Omega\times[0,\ve)$ and satisfy $J_0=1$. Hence, after possibly decreasing $\ve$, there exist constants $0<m<M<\infty$, independent of $s$, such that 
$m\le J_s(x)\le M $ for all $x \in \partial \Omega$ and all $s \in (0, \ve).$
Thus,
with the change of variable $y = F(x, s)$ $(x\in \partial\Omega)$, 
\[ 
\int_{\partial\Omega} |\partial_{\hat{\nu}}\phi(F(x,s))|^2\,dS_x = \int_{\Sigma_s} |\partial_{\hat{\nu}}\phi(y)|^2 \frac{1}{J_s(F(\cdot,s)^{-1}(y))}\,dS_y \le \frac1m \int_{\Sigma_s} |\partial_{\hat{\nu}}\phi(y)|^2\,dS_y. 
\]
Since the maps $F(\cdot,s):\partial\Omega\to\Sigma_s$ are $C^1$-diffeomorphisms and depend continuously on $s$ in the $C^1$-topology, the $C^1$-norms of these maps and of their inverses are bounded uniformly for $0<s<\ve$. Since $\partial\Omega$ is compact, we may choose finitely many coordinate charts covering $\partial\Omega$. Under the maps $F(\cdot,s)$, these charts give finitely many coordinate charts covering $\Sigma_s$, and in each such chart $\Sigma_s$ is represented as a graph 
of a $C^1$-function whose 
$C^1$-norm is bounded by a constant independent of $s$. Hence, 
the usual trace theorem on Lipschitz graphs gives a constant $c>0$, independent of $s$, 
such that 
\[
\|v\|_{L^2(\Sigma_s)}
\le c\|v\|_{H^1(\Omega)}
\qquad \forall v\in H^1(\Omega),\quad 0<s<\ve.
\]
Applying this estimate component-wise to $v=\partial_i\phi$, $i=1,2,3$, and summing gives
\[
\|\nabla\phi\|_{L^2(\Sigma_s)}
\le c\|\nabla\phi\|_{H^1(\Omega)}
\le c\|\phi\|_{H^2(\Omega)}.
\]
Therefore
\[
\|\partial_{\hat{\nu}}\phi\|_{L^2(\Sigma_s)}
\le c\|\phi\|_{H^2(\Omega)}
\quad 0<s<\ve.
\]
Consequently,
\[
\|T_t\phi \|_{L^2(\partial\Omega)}^2
\le \frac{1}{mt}\int_0^t \|\partial_{\hat{\nu}}\phi \|_{L^2(\Sigma_s)}^2\,ds
\le \frac{c^2}{mt}\int_0^t \|\phi \|_{H^2(\Omega)}^2\,ds
\le C\|\phi \|_{H^2(\Omega)}^2,
\]
with $C=c^2/m$. 
By the density of $C^{\infty}(\overline{\Omega})$ in $H^2(\Omega)$, $T_t$ extends uniquely to a bounded linear map
\begin{equation}
\label{Tt}
T_t:H^2(\Omega)\to L^2(\partial \Omega)\quad \mbox{with} \quad 
\|T_t\phi\|_{L^2(\partial \Omega)}\le \sqrt{C}\|\phi \|_{H^2(\Omega)} \qquad \forall \phi \in H^2(\Omega).
\end{equation}

Now let $w_k\in C^\infty(\overline{\Omega})$ be such that $w_k\to w$ in $H^2(\Omega)$. For each fixed $k \ge 1$, 
since $\| w_k \|_{C^2(\overline{\Omega})} < \infty$, we have by Taylor's expansion
that $T_tw_k\to \partial_{\hat{\nu}}w_k$
uniformly on $\partial \Omega$ as $t\searrow 0.$
 Also, $\partial_{\hat{\nu}}w_k\to \partial_{\hat{\nu}}w$ in $L^2(\partial \Omega).$ Therefore,
by \eqref{Tt} with $\phi$ replaced by $w - w_k$, we have
\begin{align*}
\|T_tw-\partial_{\hat{\nu}}w\|_{L^2(\partial \Omega)}
&\le
\|T_t (w-w_k)\|_{L^2(\partial\Omega)}
+\|T_tw_k-\partial_{\hat{\nu}}w_k\|_{L^2(\partial \Omega)}
+\|\partial_{\hat{\nu}}(w_k-w)\|_{L^2(\partial \Omega)}
\\
&\le
\sqrt{C}\|w-w_k\|_{H^2(\Omega)}
+\|T_tw_k-\partial_{\hat{\nu}}w_k\|_{L^2(\partial \Omega)}
+\|\partial_{\hat{\nu}}(w_k-w)\|_{L^2(\partial \Omega)}.
\end{align*}
Taking $\limsup$ as $t\searrow 0$ and then letting $k\to\infty$, we obtain $T_tw\to \partial_{\hat{\nu}}w$ in $L^2(\partial \Omega)$. 
Passing to a subsequence $t_n\searrow 0$, we have  
$T_{t_n}w\to \partial_{\hat{\nu}}w$ a.e.\ on $\partial \Omega$. 
Finally, since $w \in C(\overline{\Omega})$, $w=0$ on $\partial \Omega$, and $w\ge 0$ in $\overline{\Omega}$, we have
\[
T_{t_n} w(x)=\frac{w(F(x,t_n))-w(x)}{t_n}=\frac{w(F(x,t_n))}{t_n}\ge 0
\qquad \forall x\in\partial \Omega,\ n \ge 1. 
\]
Therefore, 
 we conclude that $\partial_{\hat{\nu}}w\ge 0$ a.e. on $\partial \Omega,$ and $ \partial_\nu w=-\partial_{\hat{\nu}}w\le 0$
a.e. on $\partial \Omega$. 
\end{proof}

The following theorem is itself of interest and will also be used later:

\begin{theorem}
    \label{t:uv}
    Let $u, v \in L^2_+(\Omega)$ be such that $0 \le u \le v$ a.e.\  in $\Omega.$
Let $\rho_u = \cT_g(u), \rho_v = \cT_g(v) \in H_g^1(\Omega)\cap H^2(\Omega)$ be given as in Theorem~\ref{t:rhou}. Then, 
$\rho_u \ge \rho_v$  in $\Omega$, $\partial_\nu \rho_u \le \partial_\nu\rho_v$ a.e.\
on $\partial \Omega$,
where $\nu$ is the unit normal at $\partial \Omega$ exterior to 
$\Omega$, and $J[u] \le J[v]$.
\end{theorem}

\begin{proof}
For each $k \ge 1,$ let $u_k = \min \{ u, k \}$, $\rho_{u_k} = \cT_g(u_k)$,  and 
 $\eta_k = \rho_{u_k} - \rho_v \in H_0^1(\Omega)\cap H^2(\Omega).$ 
By \eqref{BVP4rho} with $u$ replaced by $u_k$ and $v$, respectively, 
we get 
\begin{equation}
\label{uv}
-\Delta \eta_k + \sigma \eta_k + u_k \alpha_k \eta_k = (v-u_k) f(\rho_v) \qquad \mbox{a.e.\ in }\,  \Omega, 
\end{equation}
where $\alpha_k = 1$ for the linear case and $\alpha_k = 1/((\rho_{u_k}  + 1)(\rho_v+1))$ for the 
nonlinear case. Since $v\ge u \ge u_k\ge 0$, $f(\rho_{u_k}) \ge 0$, and $f(\rho_v) \ge 0$ a.e.\  in $\Omega,$ 
and since $u_k \alpha_k \in L^\infty(\Omega)$, it
follows from the Weak Maximum Principle (cf.\ Theorem~8.1 of \cite{GilbargTrudinger98})
applied to the operator $L = \Delta - \sigma - u_k \alpha_k $ that 
$\eta_k \ge 0,$ i.e., $\rho_{u_k} \ge \rho_v$ in $\Omega$. By Lemma~\ref{l:rhouk}, 
$\rho_{u_k} \to \rho_u$ in $H^1(\Omega)$.
Thus, $\rho_u \ge \rho_v$  in $\Omega.$ 

Set $w=\rho_u-\rho_v$. Then $w \in H^2(\Omega)\cap H_0^1(\Omega)$ and $w \geq0$ a.e. in $\Omega$. Thus, 
we have by Lemma~\ref{l:wnu} that
$\partial_\nu w \le 0$ a.e. on $\partial \Omega$. Hence, $\partial_\nu \rho_u \le \partial_\nu\rho_v$ a.e.\
on $\partial \Omega$.

It now follows from \eqref{BVP4rho} for $\rho = \rho_u$ and that for $\rho = \rho_v$, 
the fact that $ \rho_u \ge \rho_v $ in $\Omega$,  and the Divergence Theorem that 
\begin{align*}
J[u ] - J[v] &= \int_\Omega \left[ uf(\rho_u) - vf(\rho_v) \right]\,dx\\
&= \int_\Omega (\Delta \rho_u - \Delta \rho_v ) \, dx - \int_\Omega \sigma ( \rho_u - \rho_v ) \, dx 
\\
&  = \int_{\partial \Omega} (\partial_\nu \rho_u - \partial_\nu \rho_v)\, dS 
- \sigma \int_\Omega ( \rho_u - \rho_v)\, dx
\\
& \le 0. 
\end{align*}
Hence, $J[u] \le J[v].$ 
\end{proof}

\section{ 
Maximizing Sequences and Non-Existence of Maximizers}
\label{s:Local}

In this section, we show that the reaction functional $J: \cE \to \R$ 
does not admit a global maximizer. We
construct a reaction-maximizing family of localized concentrations 
$u_\ve \in \cE$ $(0 < \ve \ll 1)$ 
to obtain the supremum of the reaction functional $J$. To show the essence of 
such constructions, we first consider a one-dimensional setting and then the three-dimensional setting, which involves additional technical details.
Our constructions are only for the linear case. The supremum of 
the reaction functional $J$ for the nonlinear case is then obtained
by a comparison argument. Recall that $g_{\rm max} = \max_{x \in \partial\Omega} g(x) > 0$ as defined in \eqref{gmax}. 

\begin{theorem}
    \label{t:supJ}
    We have 
    $    \sup_{u\in \cE} J[u] = M_{\rm e} f(g_{\rm max} ). $
Moreover, the supremum is not attained in $\cE$.
\end{theorem}

\begin{proof}
Let $u \in \cE$ and let $\rho_u=\cT_g(u)$ be given as in 
Theorem~\ref{t:rhou}. We have $0 \le \rho_u < g_{\rm max}$ by Theorem~\ref{t:rhou}. 
Note that $f(s)\ge 0$ is an increasing function for $s \ge 0$. Thus, by
the definition of $\cE$, we obtain that 
\[
0 \le J[u] = \int_\Omega u f(\rho_u) \, dx \le \int_\Omega u f(g_{\rm max}) \, dx = M_{\rm e} f(g_{\rm max}).
\]
By Theorem~\ref{t:uve} below, 
there exist $u_k \in \cE$ such that $J[u_k] \to M_{\rm e} f(g_{\rm max})$ as 
$k \to \infty.$
Thus, $\sup_{v\in \cE} J[v] = M_{\rm e} f(g_{\rm max} ). $

Suppose there existed $u \in \cE$ such that $J[u]=\sup_{v\in \cE}J[v], $ which is 
$M_{\rm e} f(g_{\rm max}).$
Then, 
\begin{equation}
\label{MFJu}
\int_\Omega u [ f (g_{\rm max}) - f(\rho_u)] \, dx = M_{\rm e} f(g_{\rm max}) - J[u] = 0.
\end{equation}
Note that $u \ge 0$ as $u \in \cE$. By Theorem~\ref{t:rhou}, 
$0 \le \rho_u < g_{\rm max}$, hence $f(g_{\rm max}) > f(\rho_u)$, in $\Omega$.  
Therefore, \eqref{MFJu} implies $u = 0$ a.e.\ in $\Omega$. This contradicts the fact that
\[
\int_\Omega u \, dx =M_{\rm e}>0.
\]
Hence the supremum of $J$ over $\cE$ is not attained. 
\end{proof}

\begin{theorem}
    \label{t:uve}
There exists a family of concentrations $u_\ve\in \cE$ with $0 < \ve \le \ve_0 $ for 
some $\ve_0>0$ such that, with $\rho_\ve=\cT_g(u_\ve)$ as in 
Theorem~\ref{t:rhou}, 
\[
J[u_\ve]=\int_\Omega u_\ve f(\rho_\ve)\,dx \to M_{\rm e} f(g_{\rm max}) 
\qquad \mbox{as } \ve \to 0^+.
\]
\end{theorem}

Before proving this theorem, we first consider a one-dimensional analogue and carry out explicit calculations to construct a maximizing sequence. 
Let $L > 0$ and define 
\[
\cE_1 = \left\{ u \in L^2(0, L): u \ge 0 \mbox{ a.e.\ in }\, (0, L)
\mbox{ and } \int_0^L u(x) \, dx = M_{\rm e} \right\}. 
\]

\begin{proposition}
    \label{p:1D}
    Let $0 \le\rho_L < \rho_0$. For $0 < \ve < L$, let $u_\ve \in \cE_1$ be given by 
    $u_\ve (x)= M_{\rm e}/\ve$ if $0 \le x \le \ve$ and $u_\ve(x) = 0$ if $\ve < x \le L.$
\begin{compactenum}
    \item[\rm (1)]
    Let $\rho_\ve  = \rho_{u_\ve} \in H^1(0,L) $ be the  solution to the boundary-value problem
\begin{equation}
\label{BVP4rhove}
- \rho_\ve'' + u_\ve \rho_\ve=0  \quad \mbox{in } (0, L)
\qquad \mbox{and} \qquad \rho_\ve(0) = \rho_0, \quad  \rho_\ve (L) = \rho_L.
\end{equation}
    We  have 
    \[
    \lim_{\ve \to 0+} \int_0^L u_\ve \rho_\ve  \, dx = M_{\rm e} \rho_0. 
    \]
    \item[\rm (2)]
    Let $\hat{\rho}_\ve 
    \in H^2(0, L)$ be the unique solution
to the boundary-value problem
\begin{equation}
\label{hatrhove}
    - \hat{\rho}_\ve'' + \sigma \hat{\rho}_\ve + u_\ve f(\hat{\rho}_\ve) = 0
    \quad \mbox{in } (0, L) \qquad \mbox{and} \qquad \hat{\rho}_\ve (0) = \rho_0, \quad
    \hat{\rho}_\ve (L) = \rho_L. 
\end{equation}
    We  have 
    \[
    \lim_{\ve \to 0+} \int_0^L u_\ve f(\hat{\rho}_\ve ) \, dx = M_{\rm e} f(\rho_0). 
    \]
\end{compactenum}
\end{proposition}

\begin{proof}
(1) 
Since $u_{\ve}\rho_{\ve}\in L^2(0,L)$, we have $\rho_{\ve}\in H^2(0,L) \hookrightarrow C^1([0,L])$.
Solving \eqref{BVP4rhove}, we obtain 
\begin{equation}\label{rhove}
\rho_\ve(x) =
\begin{cases}
A_{\ve} \exp (k_{\ve} x) + (\rho_0 - A_\ve) 
\exp(-k_{\ve}x) \quad & \mbox{if } \, 0 \le x \le \ve, \\
C_{\ve} x +\rho_L-C_{\ve}L \quad & \mbox{if } \, \ve < x \le L, 
\end{cases}
\end{equation}
where 
\begin{align*}
k_{\ve} = \sqrt{\frac{M_{\rm e}}{\ve}},
\quad A_{\ve} = 
\frac{\rho_0 e^{-k_{\ve}\ve}\bigl(1-k_{\ve}(L-\ve)\bigr)-\rho_L}
{2\left[\cosh(k_{\ve}\ve)\bigl(1-k_{\ve}(L-\ve)\bigr)-e^{k_{\ve}\ve}\right]},
\quad
C_{\ve} =
\frac{2k_{\ve}A_{\ve}e^{k_{\ve}\ve}-k_{\ve}\rho_L}
{1-k_{\ve}(L-\ve)}.
\end{align*}
Routine calculations show that  
$A_{\ve} \to  {\rho_0}/{2}$ as $\ve \to 0^+.$ Thus, if 
$\ve > 0$ is small enough, $C_\ve < 0$ and the solution $\rho_\ve > 0$ on $(0, L).$
Moreover, 
\begin{align*}
\int_0^{L} u_{\ve}\rho_{\ve} \, dx 
&=\frac{M_{\rm e}}{ \ve} \int_0^{\ve} \left[ A_\ve  e^{k_{\ve} x} + (\rho_0-A_\ve) e^{-k_{\ve}x}\right] dx \\
&= M_{\rm e}\left[A_\ve \left(\frac{e^{k_{\ve}\ve}-1}{k_{\ve}\ve}\right) + (\rho_0-A_\ve) \left(\frac{e^{-k_{\ve}\ve}-1}{-k_{\ve}\ve}\right)\right] 
\\
&
\to 
M_{\rm e} \rho_0
\qquad \mbox{as } \ve \to 0^+. 
\end{align*}

(2) 
On $(\ve, L)$, $\hat{\rho}_\ve$ satisfies
$
-\hat{\rho}_\ve''+\sigma\hat{\rho}_\ve=0.
$
Solving this equation with the boundary conditions $\hat{\rho}_\ve(\ve)=:a_\ve$ and $\hat{\rho}_\ve(L)=\rho_L$, 
we get 
\[
\hat{\rho}_\ve(x)= \frac{\rho_L\sinh\left(\sqrt\sigma(x-\ve)\right)
+a_\ve\sinh\left(\sqrt\sigma(L-x)\right) 
}
{\sinh\left(\sqrt\sigma(L-\ve)\right)} \quad \forall x \in (\ve,L). 
\]
Differentiating, we get
\begin{align*}
|\hat{\rho}_\ve'(x)| & \le \sqrt{\sigma} \left[ \frac{\rho_L \cosh \left(\sqrt{\sigma}(x-\ve)\right)+a_\ve \cosh \left(\sqrt{\sigma}(L-x)\right)}{\sinh\left(\sqrt{\sigma}(L-\ve)\right)} \right]\\
& \le \sqrt{\sigma} \left[ \frac{\rho_L \cosh\left( \sqrt{\sigma}(L-\ve)\right)+a_\ve \cosh \left( \sqrt{\sigma}(L-\ve)\right)}{\sinh\left(\sqrt{\sigma}(L-\ve)\right)} \right]\\
&\le 2\sqrt{\sigma}\rho_0 \coth \left(\sqrt{\sigma}(L-\ve)\right)\qquad \forall x \in (\ve, L), 
\end{align*}
since $0\le \hat{\rho}_\ve \le \rho_0$ in $(0,L)$ by 
the one-dimensional analogue of Theorem~\ref{t:rhou}. Thus, for $0<\ve < L/2$, we have 
\begin{equation}
\label{hatrhovep}
|\hat{\rho}_\ve'(x)|\le 2\sqrt{\sigma}\rho_0 \coth \left(  \frac{L\sqrt{\sigma}}{2} \right) \qquad \forall x \in [\ve,L].
\end{equation}

By the equation for $\hat{\rho}_\ve$ in \eqref{hatrhove}, we get
\[
\hat{\rho}_\ve'(\ve)-\hat{\rho}_\ve'(0)
=\int_0^\ve \hat{\rho}_\ve''(s)\,ds= \sigma\int_0^\ve \hat{\rho}_\ve(s)\,ds
  + \frac{M_{\rm e}}{\ve}\int_0^\ve f(\hat{\rho}_\ve(s))\,ds.
\]
Consequently, since $0 \le \hat{\rho}_\ve \le \rho_0$, we have for $0 < \ve < L/2$ that 
\begin{align*}
|\hat{\rho}_\ve'(\ve)-\hat{\rho}_\ve'(0)|
& \le \sigma\int_0^\ve \rho_0\,ds
  + \frac{M_{\rm e}}{\ve}\int_0^\ve f(\rho_0)\,ds 
\le \frac12 \sigma \rho_0 L + M_{\rm e} f(\rho_0). 
\end{align*}
This and \eqref{hatrhovep} imply that 
$| \hat{\rho}'_\ve (0) | \le C $ and 
$ | \hat{\rho}'_\ve(\ve) | \le C$ if  $ 0 < \ve < L/2,$
where $C > 0$ is a constant independent of $\ve \in (0, L/2).$
By \eqref{hatrhove}, $\hat{\rho}_\ve''\ge 0$ a.e. on $(0,L)$. Thus, 
$\hat{\rho}_\ve'(0)\le \hat{\rho}_\ve'(x)\le \hat{\rho}_\ve'(\ve) $ for all 
$x \in (0, \ve)$ and all $\ve \in (0, L/2).$ Hence, 
$\sup_{0\le x \le \ve} | \hat{\rho}'_\ve (x)| \le C $
for all $ \ve \in (0, L/2).$
Consequently, 
\[
\sup_{0 \le x \le \ve} \left| \hat{\rho}_\ve(x)-\rho_0\right|
  = \sup_{0 \le x \le \ve} \left|\int_0^x \hat{\rho}_\ve'(s)\,ds\right|  
  \le C \ve
  \qquad \forall \ve \in (0, L/2). 
\]

Finally, by the definition of $u_\ve $ and the fact that 
$|f'(s) | \le 1$ for all $s \ge 0$, we have 
\begin{align*}
    \left| \int_0^L u_\ve f (\hat{\rho}_\ve) \, dx - M_{\rm e} f (\rho_0) \right| 
    &= \left| \int_0^\ve \frac{M_{\rm e}}{\ve} \left[ f(\hat{\rho}_\ve) - f(\rho_0) \right] \, dx \right|\\
    & \le \frac{M_{\rm e}}{\ve} \int_0^\ve | \hat{\rho}_\ve(x) - \rho_0 | \, dx 
     \\
     &\le M_{\rm e} C \ve \to 0 \qquad \mbox{as } \ve \to 0. 
\end{align*}
This completes the proof. 
\end{proof}

We now prove a localization result that will be used in our proof of Theorem~\ref{t:uve}. 

\begin{lemma}\label{l:construction}
For any $\ell > 0$, we define $\psi_\ell:[0, \infty) \to \R $ by
  \[
    \psi_\ell(s):=
    \begin{cases}
    2\ell s-s^2 & \quad \mbox{if } \, 0\le s\le \ell,\\[1mm]
    \ell^2 & \quad \mbox{if }\, s\ge \ell,
    \end{cases}
    \]
    and $W_\ell = \psi_\ell \circ d: \Omega \to \R$, where 
 $d(x):=\dist(x,\partial\Omega)$ $(x \in \overline{\Omega})$. 
There exist $\delta>0$ and $K>0$ such that, 
    for $0 < \ell \le \min\{\delta/2,1/(2K)\}$, 
    the following hold true:
\begin{compactenum}
    \item[{\rm (1)}] 
$  W_\ell\in H_0^1(\Omega)\cap W^{1,\infty}(\Omega)$, $ 0\le W_\ell\le \ell^2$ 
in $\Omega$, and $-\Delta W_\ell\ge \chi_{\{d<\ell\}}$ in $\Omega$ in the weak sense; 
    
    \item[{\rm (2)}] 
    If $E\subset \{x\in\Omega:d(x)<\ell\}$
    is measurable and $T_E\in H_0^1(\Omega)$ is the unique weak solution of $-\Delta T_E=\chi_E$ in $\Omega$,
    then $ 0\le T_E\le W_\ell\le \ell^2$ in $\Omega.$
\end{compactenum}
\end{lemma}

\begin{proof} 
We have $d\in C^2\bigl(\{x\in\overline{\Omega}:d(x)<\delta\}\bigr)$ for some $\delta > 0$
small enough; cf.\  Lemma $14.16$ in \cite{GilbargTrudinger98}. Shrinking $\delta$ if necessary, 
we have 
\begin{equation} 
\label{ddK}
|\nabla d|=1 \qquad \mbox{and} \qquad |\Delta d|\le K \qquad \mbox{on } D_\delta: = \{ x \in \Omega: d(x) < \delta\},
\end{equation}
for some constant $K > 0.$

(1) Fix $0 < \ell \le \min\{\delta/2,1/(2K)\}.$ Since $d$ is Lipschitz on $\overline{\Omega}$ 
and $\psi_\ell$ is Lipschitz on $[0,\infty)$, the composition $W_\ell=\psi_\ell\circ d$
belongs to $W^{1,\infty}(\Omega)\subset H^1(\Omega)$. Moreover, since $d=0$ on $\partial\Omega$ and $\psi_\ell(0)=0$, we have $W_\ell\in H_0^1(\Omega)\cap W^{1,\infty}(\Omega).$
Also, by the definition of $\psi_\ell$, $0\le W_\ell(x)\le \ell^2$ for all $x \in \Omega.$ 

We now prove the weak inequality $-\Delta W_{\ell} \ge \chi_{\{d<\ell\}}$ in $\Omega$. Note that 
$\psi_\ell \in C^1([0, \infty))$ and
\[
\psi_\ell'(s)= 
\begin{cases}
2\ell-2s & \quad \mbox{if }\, 0\le s<\ell,\\
0 & \quad \mbox{if } \, s\ge \ell.
\end{cases}
\]
Hence, $\psi'_\ell$ is Lipschitz on $[0,\infty)$. 
The chain rule for Lipschitz compositions yields $\nabla W_\ell=\psi_\ell' (d)\nabla d$
a.e.\ in $\Omega$. 
Set $D_\ell=\{x\in\Omega:d(x)<\ell\}$. 
Let $\varphi\in H_0^1(\Omega)$. 
Since $\nabla W_\ell=0$ a.e. on $\Omega\setminus D_\ell$, we have
\begin{equation}
\label{Well}
\int_\Omega \nabla W_\ell\cdot \nabla\varphi\,dx
=
\int_{D_\ell}\psi'_\ell(d)\nabla d\cdot \nabla\varphi\,dx.
\end{equation}

Define $V_\ell(x)=\psi'_\ell(d(x))\nabla d(x)$ $(x\in \overline{D_\ell}).$ 
Since $d\in C^2(\overline{D_\ell})$, we have $V_\ell\in C^1(\overline{D_\ell};\R^3).$ Also, 
$D_\ell$ is a $C^2$ domain, since  $\partial\Omega$ is $C^2$ 
and $\{ x \in \Omega: d=\ell\}$ is a $C^2$ 
level surface because $|\nabla d|=1\neq 0$ there. Therefore, the divergence theorem gives
\begin{equation} 
\label{Dell} 
\int_{D_\ell}V_\ell\cdot \nabla\varphi\,dx
=
-\int_{D_\ell} (\operatorname{div}V_\ell)\,\varphi\,dx
+
\int_{\partial D_\ell}(V_\ell\cdot \nu_{D_\ell})\,\varphi\,dS.
\end{equation}
If $x \in \partial D_\ell$, then either $d(x) = \ell$ and $\psi'_\ell(d(x)) = 0$, or $x\in \partial \Omega$ where $\varphi=0$. Thus, the boundary integral vanishes. 
If  $x\in D_\ell$, then $d(x) \le \ell$ and $|\nabla d(x)|=1$. Hence
\begin{align*}
\operatorname{div} V_\ell (x) &= \operatorname{div}\!\bigl(\psi_\ell'(d(x))\nabla d(x)\bigr)
\\
&= \psi_\ell''(d(x))|\nabla d(x)|^2+\psi'_\ell(d (x))\Delta d(x)
\\
&=
-2+(2\ell-2d(x))\Delta d(x) \qquad \mbox{if }\, x \in D_\ell.
\end{align*}
Define \[
f_\ell(x):=
\begin{cases}
2-(2\ell-2d(x))\Delta d(x), & \mbox{if } x\in D_\ell,\\[1mm]
0, & \mbox{if } x\in \Omega\setminus D_\ell,
\end{cases}
\]
Then, $f_\ell\in L^\infty(\Omega)$, and by \eqref{Well} and
\eqref{Dell}, 
\begin{equation} 
\label{fell}
\int_\Omega \nabla W_\ell\cdot \nabla\varphi\,dx
=
\int_\Omega f_\ell\,\varphi\,dx
\qquad
\forall \varphi\in H_0^1(\Omega).
\end{equation}

We next estimate $f_\ell$. On $D_\ell$, we have $0\le d<\ell\le \delta$, and $|\Delta d|\le K$ by \eqref{ddK}. So, 
\[
f_\ell
=
2-(2\ell-2d)\Delta d
\ge
2-(2\ell-2d)|\Delta d|
\ge
2-2\ell K
\ge 1 \qquad \mbox{on }\, D_\ell. 
\]
Therefore $f_\ell\ge \chi_{D_\ell}=\chi_{\{d<\ell\}}$ a.e.\ in $\Omega.$
Hence, for every nonnegative $\varphi\in H_0^1(\Omega)$, by \eqref{fell} we have that 
\[
\int_\Omega \nabla W_\ell\cdot \nabla\varphi\,dx
=
\int_\Omega f_\ell\,\varphi\,dx
\ge
\int_\Omega \chi_{\{d<\ell\}}\,\varphi\,dx, 
\]
as desired. 

(2) 
By the standard elliptic theory, we have $T_E\ge 0$
a.e.\ in $\Omega.$ Let $E \subset \{ x \in \Omega: d(x) < \ell\}$ with
$0 < \ell \le \min\{\delta/2,1/(2K)\}. $
Since $-\Delta W_\ell \ge \chi_{\{d<\ell\}}\ge \chi_E= -\Delta T_E$ in the weak sense, 
standard elliptic theory gives $T_E\le W_\ell $
a.e.\ in $\Omega.$ Combining these inequalities with 
$0\le W_\ell\le \ell^2$ yields $0\le T_E\le W_\ell\le \ell^2$ 
a.e.\ in $\Omega.$
\end{proof}
We are now ready to prove the main result. 

\begin{proof}[Proof of Theorem~\ref{t:uve}]
We construct a family of subsets $S_\ve\subset \overline{\Omega}$ for all $\ve > 0$
that are sufficiently small and define the related functions $u_\ve \in \cE$ by
\begin{equation}
    \label{uve} u_\ve = \frac{M_{\rm e}}{|S_\ve|} \chi_{S_\ve} \qquad \mbox{in } \Omega. 
\end{equation}
The set $S_\ve$ is in an $O(\ve)$-neighborhood of $y_0 \in \partial \Omega$ at which the boundary
value $g$ achieves its maximum value $g_{\rm max} > 0$. 
Let $\rho_\ve = \cT_g(u_\ve)$ with the linear case
(i.e., $f(s) = s$ for any $s$); cf.\ Theorem~\ref{t:rhou}. We shall prove that
\begin{equation}
    \label{rhove2g}
    \lim_{\ve \to 0^+} \sup_{x \in S_\ve } |\rho_\ve (x) - g_{\rm max}| = 0. 
\end{equation}

Suppose this is correct. Then, for the linear case, we have 
\begin{align*}
\left| J [u_\ve] - M_{\rm e} f(g_{\rm max}) \right| &= 
\left| \int_\Omega u_\ve \rho_\ve \,dx - M_{\rm e} g_{\rm max}\right|
\\
& = \frac{M_{\rm e}}{|S_\ve|} \left| \int_{S_\ve} (\rho_\ve - g_{\rm max} )\, dx \right|
\\
&\le M_{\rm e} \sup_{x \in S_\ve} |\rho_\ve (x) - g_{\rm max} | \\
&\to 0 
\qquad \mbox{as } \ve \to 0^+. 
\end{align*}
Hence, $J[u_\ve]\to M_{\rm e} f(g_{\rm max})$ as $\ve \to 0^+.$
Consider the nonlinear case: $f(s) = s/(s+1)$ for $s > -1.$ 
Let $\hat{\rho}_\ve = \cT_g(u_\ve) $ (cf.\ Theorem~\ref{t:rhou}). 
Since $u_\ve f(\hat{\rho}_\ve) = v_\ve \hat{\rho}_\ve$ with 
$ v_\ve = u_\ve / (\hat{\rho}_\ve + 1) \le u_\ve $ in $\Omega,$
it follows from Theorems~\ref{t:rhou} and \ref{t:uv} that $ 0 \le \rho_\ve
\le \hat{\rho}_\ve \le g_{\rm max}$ in $\Omega.$
Hence, \eqref{rhove2g} holds true with $\rho_\ve$ replaced by $\hat{\rho}_\ve.$
Consequently, since $|f(s) - f(t)| \le | s - t|$ for $s, t > 0$, we have 
\begin{align*}
\left| J [u_\ve] - M_{\rm e} f(g_{\rm max}) \right| &= 
\left| \int_\Omega u_\ve f(\hat{\rho}_\ve) \,dx - M_{\rm e} f(g_{\rm max})\right|
\\
& \le \frac{M_{\rm e}}{|S_\ve|}  \int_{S_\ve} | \hat{\rho}_\ve - g_{\rm max} | \, dx
\\
&\le M_{\rm e} \sup_{x \in S_\ve} |\rho_\ve (x) - g_{\rm max} |
\\
&\to 0 
\qquad \mbox{as } \ve \to 0^+. 
\end{align*}
Thus, 
$J[u_\ve]\to M_{\rm e} f(g_{\rm max})$ as $\ve \to 0^+.$ 

It remains to construct $S_\ve$ and prove \eqref{rhove2g} for the linear case. 
We proceed in five steps. 

{\it {Step 1:  Construction of $S_\ve:$ a concentrating boundary cap.}}
Because $\partial\Omega$ is of class $C^2$, there exists 
$\delta_0>0$ such that the inward normal coordinate map
$
F:\partial\Omega\times(0,\delta_0)\to \Omega,$ defined by 
\[
F(z,t)=z-t\nu(z) \qquad \forall (z, t) \in \partial\Omega \times (0, \delta_0),
\]
is a $C^1$-diffeomorphism onto the tubular neighborhood $\{x\in\Omega:0<d(x)<\delta_0\},$
where $\nu(z)$ denotes the outer unit normal at $z\in\partial\Omega$.
For any $\ve > 0$ with $ \ve^3 < \delta_0$, we define
\[
\Gamma_\ve:=\partial\Omega\cap B_{\ve} (y_0) \qquad \mbox{and} 
\qquad
S_\ve:=\{z-t\nu(z): z\in\Gamma_\ve,\ 0<t< \ve^3\}; 
\]
cf.\ Figure~\ref{f:Sve}.
It follows that there exist constants $c,C>0$, independent of $\ve$,  such that
\begin{equation}\label{Sve}
c \ve^2 \le |\Gamma_\ve|_{\partial \Omega} \le C \ve^2 \quad \mbox{and} \quad 
c\,\ve^5\le |S_\ve|\le C\,\ve^5 \qquad \mbox{if } \ve > 0 \ \mbox{is small enough,}
\end{equation}
where $ | \cdot |_{\partial \Omega}$ denotes the surface measure. 

\begin{figure}[h]
\begin{center}
\includegraphics[width=2.4in,height=2.6in]{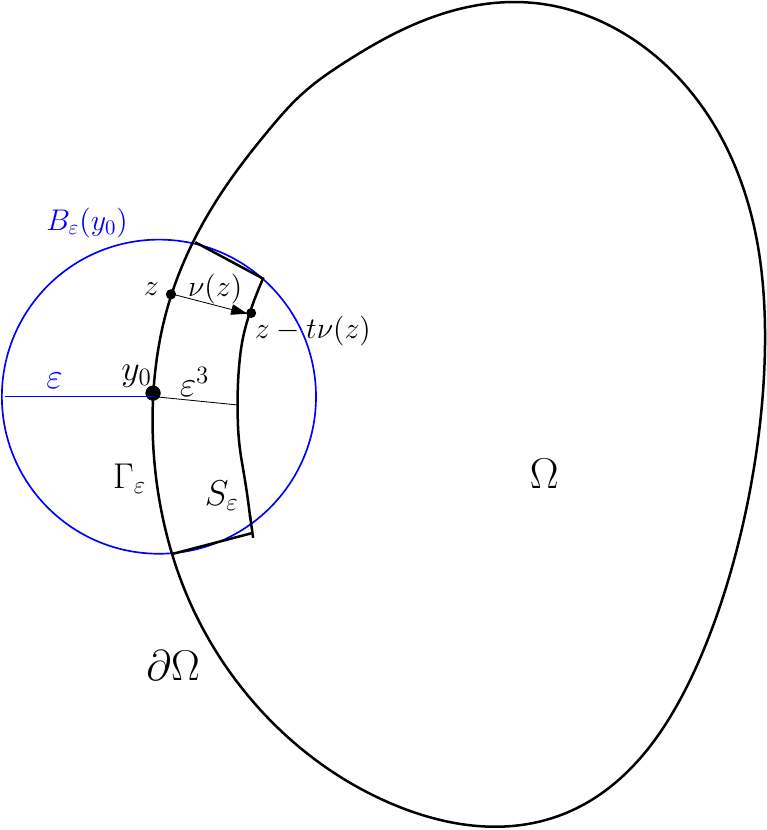}
\end{center}
\vspace{-4 mm}

\caption{Schematic of the construction of $S_\ve.$}
\label{f:Sve}
\end{figure}

{\it{Step 2: Approximation of $\rho_\ve$.}} Let $h\in H^2(\Omega)\hookrightarrow C(\overline{\Omega}) $ be defined by 
$\Delta h=0$ in $\Omega$ and $h = g$ on $\partial\Omega.$ We have 
 $0 \le h \le g_{\rm max}$ in $\Omega.$ 
Since $S_\ve\subset B_{\ve+\ve^3} (y_0)$,  we have 
\begin{equation}\label{h2g}
\sup_{x\in S_\ve} | h(x) - g_{\rm max} | 
\le \sup_{x \in S_\ve} | h(x) - h(y_0)| \to 0
\qquad \mbox{as } \ve \to 0^+.
\end{equation}
Define 
\[
w_\ve=h-\rho_\ve \in H_0^1(\Omega) \qquad \mbox{and} \qquad q_\ve=\sigma + u_\ve \in L^2(\Omega).
\]
We further decompose $w_\ve$ into
\begin{equation}
    \label{wabve}
    w_\ve = a_\ve + b_\ve \qquad \mbox{in } \Omega, 
\end{equation}
where $a_\ve,b_\ve\in H_0^1(\Omega)$ are the unique weak solutions of
\begin{align}
    \label{ave}
&-\Delta a_\ve + q_\ve a_\ve
= \sigma h
\qquad\mbox{in }\Omega,\\
\label{bve}
&-\Delta b_\ve + q_\ve b_\ve
=u_\ve h
\qquad\mbox{in }\Omega, 
\end{align}
respectively. 
Since $\sigma h, u_\ve h \in L^2_+(\Omega)$, we have
$a_\ve, b_\ve \ge 0$ in $\Omega$ and 
$a_\ve,\ b_\ve \in H^2(\Omega)\hookrightarrow C(\overline{\Omega}).$

{\it{Step 3: Estimate $a_\ve$.}} Let $\tau\in H_0^1(\Omega)$ be the unique weak solution of
$-\Delta\tau=1$ in $\Omega$. We have $\tau \in W^{2, p}(\Omega)$ for any $p \in [1, \infty)$, 
and hence $\tau \in C^{1}(\overline{\Omega})$, 
and $\tau \ge 0 $ in $\Omega.$ 
Define $\overline{a}=\sigma g_{\rm max}\tau.$ Since $q_\ve \ge 0$ and $0\le h\le g_{\rm max}$, we have
\[
-\Delta\overline{a} + q_\ve\overline{a}
=
\sigma g_{\rm max} + q_\ve \sigma g_{\rm max} \tau
\ge \sigma g_{\rm max}
\ge \sigma h \qquad \mbox{in } \Omega. 
\]
Thus, $\overline{a}$ is a 
supersolution of \eqref{ave}. 
Hence, 
\begin{equation}\label{ave1}
0\le a_\ve\le \overline{a} =\sigma g_{\rm max}\tau
\qquad\mbox{in }\Omega.
\end{equation}
Note that $\tau$ is Lipschitz-continuous on $\overline{\Omega}$. Let $L\ge 0$ be the Lipschitz constant. 
Fix $x\in\Omega$ and choose $z\in\partial\Omega$ such that $|x-z|=d(x).$ Since $\tau(z)=0$, we have 
\[\tau(x)=|\tau(x)-\tau(z)|
\le L|x-z|
=
Ld(x). 
\]
Therefore, 
$\tau(x)\le Ld(x)$ for all $x \in \Omega.$
If $x\in S_\ve$, then $d(x)\le \ve^3,$ and 
hence $\tau(x) \le L\ve^3.$
This and \eqref{ave1} imply that 
\begin{equation}\label{ave2}
\sup_{x\in S_\ve}a_\ve(x)\le \sup_{x \in S_\ve} 
\sigma g_{\rm max} \tau(x) \le \sigma g_{\rm max} L \ve^3 \to 0 \qquad \mbox{as } \ve \to 0.
\end{equation}

{\it{Step 4: Estimate $b_\ve$.}}
Let $T_\ve\in H_0^1(\Omega)$ be the unique weak solution of
$-\Delta T_\ve=\chi_{S_\ve}$ in $\Omega.$ We have $T_\ve \in H^2(\Omega)\hookrightarrow C(\overline{\Omega}).$
Since $S_\ve\subset\{x\in\Omega:d(x)<\ve^3\}$
for all sufficiently small $\ve>0$, we have by 
Lemma~\ref{l:construction} with $E = S_\ve$ and $\ell = \ve^3$ that 
\begin{equation}\label{Tve}
0\le T_\ve\le \ve^6
\qquad\mbox{in } \Omega. 
\end{equation}
Define $\overline{b}_\ve=g_{\rm max}\alpha_\ve T_\ve,$ where $\alpha_\ve = M_{\rm e}/|S_\ve|$.
Since $u_\ve=\alpha_\ve\chi_{S_\ve}$ and $0\le h\le g_{\rm max}$,
we have 
\[
-\Delta\overline{b}_\ve + q_\ve\overline{b}_\ve
=
g_{\rm max}\alpha_\ve\chi_{S_\ve}
+
q_\ve g_{\rm max}\alpha_\ve T_\ve
\ge
g_{\rm max}\alpha_\ve\chi_{S_\ve}
\ge
u_\ve h \qquad \mbox{in } \Omega. 
\]
Therefore, $\overline{b}_\ve$ is a
supersolution of \eqref{bve}. 
Hence $0\le b_\ve\le \overline{b}_\ve $ 
in $\Omega.$
By \eqref{Sve} and \eqref{Tve}, we thus obtain $0\le b_\ve\le g_{\rm max}\alpha_\ve \ve^6
\le (g_{\rm max} M_{\rm e}/c)\ve$ in $\Omega.$
Hence, 
\begin{equation}\label{bve1}
\sup_{x\in S_\ve}b_\ve(x) \to 0 \qquad \mbox{as } \ve \to 0^+.
\end{equation}

{\it{Step 5: Conclusion.}} By Step 2, we have for every $x\in S_\ve$ that
\[
\rho_\ve(x) = h(x)  - w_\ve(x) 
=
h(x)-a_\ve(x)-b_\ve(x). 
\]
Hence, it follows from \eqref{h2g}, \eqref{ave2}, and \eqref{bve1} that
\[
\sup_{x \in S_\ve} | \rho_\ve (x) - g_{\rm max}|
\le \sup_{x \in S_\ve} | h(x) - g_{\rm max} | + 
\sup_{x\in S_\ve} a_\ve(x)  + \sup_{x\in S_\ve} b_\ve(x) \to 0
\qquad \mbox{as } \ve \to 0^+.
\]
This completes the proof. 
\end{proof}

\section{First and Second Variations}
\label{s:FirstVar}

Let $u \in \cE$ and define
\begin{align}
\label{cAu}
&\cA_u  = \left\{ v \in L^2(\Omega): 
u + t v \in \cE \ 
\mbox{ if } \ |t | \le t_v \ \mbox{ for some }\  t_v > 0 \right\}, 
\\
\label{cAu+}
&\cA_u^+  = \left\{ v \in L^2(\Omega): 
u + t v \in \cE \ 
\mbox{ if } \ 0 \le t \le t_v^+ \ \mbox{ for some } \ t_v^+ > 0 \right\}. 
\end{align}
Note that $\emptyset \ne\cA_u \subseteq \cA_u^+$. 
For any $v \in \cA_u$, we define the first variation 
\[
J'[u][v] = \frac{d}{dt}\biggr|_{t=0} J[ u + t v] = \lim_{t \to 0} \frac{J[u+tv]-J[u]}{t},
\]
if the derivative exists, where the functional $J: \cE \to [0, \infty)$ is defined 
in \eqref{Ju}. If $v \in \cA_u^+$, we define the positive-side first variation 
\[
J'_+[u][v] = \frac{d}{dt}\biggr|_{t=0^+} J[ u + t v] = \lim_{t \to 0^+} \frac{J[u+tv]-J[u]}{t},
\]
if the one-sided derivative exists. Clearly, if $v \in \cA_u$ and its first variation
$J'[u][v]$ exists, then $v \in \cA_u^+$ and its positive-side first variation $J'_+[u][v]$
exists and is the same as $J'[u][v].$ 

The  following lemma will be used for calculating the first and second variations: 

\begin{lemma}
    \label{l:Lu} 
    Let $u \in L^2_+(\Omega)$ 
    and $\rho_u = \cT_g(u) \in H_g^1(\Omega)$  as defined in Theorem~\ref{t:rhou}.
    \begin{compactenum}
    \item[\rm (1)]
For any $w \in L^2(\Omega)$, there exists a unique weak solution, $ \psi \in H_0^1(\Omega) \cap 
    H^2(\Omega)$, to the boundary-value problem
    \begin{equation}
    \label{BVPpsi}
    \left\{ 
    \begin{aligned}
        & - \Delta \psi + \sigma \psi + u f'(\rho_u) \psi = w & &\quad  \mbox{in } \Omega, \\
        & \psi= 0 & &\quad  \mbox{on } \partial \Omega,
    \end{aligned}
    \right.
    \end{equation}
defined by $\psi \in H_0^1(\Omega)$ and 
    \begin{align}
    \label{psiw}
        \int_{\Omega} \left[ \nabla \psi \cdot \nabla \phi + \sigma \psi \phi + u f'(\rho_u) \psi \phi \right] dx
        = \int_\Omega w \phi \, dx \qquad \forall \phi \in H_0^1(\Omega). 
    \end{align}
The operator $\cL_u = (-\Delta + \sigma + u f'(\rho_u))^{-1}: L^2(\Omega) \to H_0^1(\Omega)\cap H^2(\Omega)$ defined by 
$\cL_u w = \psi $ is linear and self-adjoint. 
\item[\rm (2)]
Let  $\psi_u = \cL_u (uf'(\rho_u))
\in H_0^1(\Omega)\cap H^2(\Omega)$ as defined in Part (1). 
Then $0 \le \psi_u \le 1$ in $\Omega$  and   $\{ \psi_u < 1 \}$ is dense in $\Omega.$
If $| \{ u > 0 \}| > 0$, then $|\{ u > 0 \} \cap \{ \psi_u < 1 \} | > 0.$
\end{compactenum}
\end{lemma}

\begin{proof}
(1) Note that $f'(\rho_u)\in L^\infty(\Omega)$ and hence $u f'(\rho_u) \in L^2(\Omega). $ Moreover, 
$u f'(\rho_u) \ge 0$ a.e.\ in $\Omega,$ $|f'(\rho_u)|\le 1$ in $\Omega$, and 
$H^1(\Omega)\hookrightarrow L^6(\Omega) \subset L^4(\Omega).$
Thus, by the direct method, we infer that there exists a unique
$\psi\in H_0^1(\Omega)$ that minimizes the functional $K_u: H_0^1(\Omega) \to \R\cup \{ \infty \}$, defined by 
\[
K_u[\eta] = \int_\Omega \left[\frac12 |\nabla \eta|^2 
+ \frac{\sigma}{2} \eta^2 + \frac12 u f'(\rho_u) \eta^2 - w\eta \right]dx. 
\]
Routine calculations imply that $\psi$ satisfies \eqref{psiw}.
Suppose $\tilde{\psi}
\in H_0^1(\Omega)$ also satisfies \eqref{psiw} with $\tilde{\psi}$ replacing $\psi.$
Let $\phi = \psi - \tilde{\psi} \in H_0^1(\Omega).$ By \eqref{psiw} and
the related equation for $\tilde{\psi}$ with this $\phi$, we can see that $\psi = \tilde{\psi}$. 
This implies the uniqueness of the weak solution. 

To show $\psi \in H^2(\Omega)$, we first notice that $u \psi \in L^{3/2}(\Omega)$, 
which follows from $u \in L^2(\Omega)$, $\psi \in H^1(\Omega)
\hookrightarrow L^6(\Omega)$, and H{\"o}lder's inequality. Consequently, since 
$f'(\rho_u) \in L^\infty(\Omega)$, $\Delta \psi \in L^{3/2}(\Omega).$ Hence, 
$\psi\in W^{2, 3/2}(\Omega)$ (cf.\ Theorem~9.15 in \cite{GilbargTrudinger98}). 
Note that $ W^{2, 3/2}(\Omega) \hookrightarrow W^{1, 3} (\Omega)\hookrightarrow L^p
(\Omega)$ for 
any $ p \in [1, \infty).$ Choose $p = 9.$  
We get $\psi \in L^9(\Omega)$ and  $uf'(\rho_u)\psi \in L^{18/11}(\Omega)$. 
Applying Theorem 9.15 in \cite{GilbargTrudinger98} again, we get $\psi \in W^{2,18/11}(\Omega) \hookrightarrow C(\overline{\Omega})$. Thus, $uf'(\rho_u)\psi \in L^{2}(\Omega)$. One final application of Theorem 9.15 in \cite{GilbargTrudinger98} gives
$\psi \in H^2(\Omega).$

It is clear that $\cL_u: L^2(\Omega) \to H_0^1(\Omega)\cap H^2(\Omega)$ is linear and self-adjoint.

(2)
Note from Part (1) that 
$\psi_u $ is the unique minimizer of $K_u: H_0^1(\Omega) \to \R\cup \{ \infty \}$, defined by
\[
K_u[\eta] = \int_\Omega \left[\frac12 |\nabla \eta|^2 
+ \frac{\sigma}{2} \eta^2 + \frac12 u f'(\rho_u) \eta^2 - u f'(\rho_u) \eta \right]dx 
\qquad \forall \eta \in H_0^1(\Omega).
\]
Define $\psi = \max \{ 0, \min \{ \psi_u, 1 \} \} \in H_0^1(\Omega)$. Clearly, $0 \le \psi \le 1$
a.e.\ in $\Omega$.  Since $ \psi_u = \psi$ on $\{ x \in \Omega: 0 \le \psi_u \le 1\}$, 
and  
$| \nabla \psi_u | \ge | \nabla \psi |$ and $u f'(\rho_u) \ge 0$ 
a.e.\ in $\Omega$, we  have 
\begin{align*}
0 & \ge K_u[\psi_u] - K_u [\psi] \nonumber \\
& \ge \int_{ \{ \psi_u <  0 \}} 
\left[ \frac{\sigma}{2} \psi_u^2  + \frac{1}{2} u f'(\rho_u) 
 \psi_u^2  - uf'(\rho_u) \psi_u  \right] dx\\
&\qquad + 
\int_{ \{ \psi_u > 1 \}} 
\left[ \frac{\sigma}{2} \left( \psi_u^2 - 1 \right) + \frac{1}{2} u f'(\rho_u) 
\left( \psi_u^2-  1 \right)  - uf'(\rho_u) ( \psi_u - 1 )  \right] dx
\nonumber \\
&   =\int_{ \{ \psi_u <  0 \}} 
\left[ \frac{\sigma}{2} \psi_u^2  + \frac{1}{2} u f'(\rho_u) \psi_u (\psi_u - 2) \right] dx
\\
&\qquad + \int_{ \{ \psi_u > 1\}} 
\left[ \frac{\sigma}{2} \left( \psi_u^2 - 1\right) 
+ \frac{1}{2} u f'(\rho_u) \left( \psi_u -1 \right)^2  \right]dx 
\nonumber \\
&\ge 0. 
\end{align*}
Thus, $K_u[\psi_u] = K_u [\psi]$ and hence  $\psi$ is also a minimizer of $K_u$ over $H_0^1(\Omega)$. The
uniqueness of such a minimizer implies that $\psi_u = \psi$. Moreover, each of 
the last two integrals above vanishes, implying that $|\{ \psi_u < 0 \}| = 0$
and $|\{ \psi_u > 1 \} | = 0.$ Thus, 
$0 \le \psi_u = \psi \le 1 $ a.e.\ in $\Omega.$ As $\psi_u \in H^2(\Omega) \hookrightarrow C(\overline{\Omega})$, this a.e. estimate upgrades to the point-wise bound $0 \le \psi_u\le 1$ in $\Omega$. 

If $\{ \psi_u < 1\}$ is not dense in $\Omega$, then there exists a non-empty open set 
$B \subset \Omega$ such that $\psi_u \ge 1$ and hence $\psi_u = 1$ on $B.$ Thus, by
\eqref{BVPpsi} with $w = u f'(\rho_u)$, we have $\sigma = 0$ in $B$, contradicting
the fact that $\sigma > 0.$

Assume $|\{ u > 0 \}| > 0$. We show that $| \{ u > 0 \} \cap \{ \psi_u < 1 \} | > 0.$
If otherwise $| \{ u > 0 \} \cap \{ \psi_u < 1 \} | = 0$, then $ \psi_u = 1 $ a.e.\ in $\{ u > 0 \}.$
Consequently, by \eqref{BVPpsi} with $\psi_u$ replacing $\psi$
and $uf'(\rho_u)$ replacing $w$, 
$-\Delta \psi_u + \sigma \psi_u = 0$ a.e.\ in $\Omega$. But $\psi_u \in H_0^1(\Omega)$. Thus, 
$\psi_u = 0$ in $\Omega$. This then implies by \eqref{BVPpsi}  with $w = u f'(\rho_u) \ge 0$ that
$ u f'(\rho_u) = 0$ a.e.\ in $\Omega$. As $f'(\rho_u) \neq 0$ in $\Omega$, we get $u = 0$ a.e.\ in $\Omega$, 
contradicting $| \{ u > 0 \} | > 0$. 
\end{proof}

We now calculate the first variation of the functional $J: \cE \to \R.$

\begin{theorem}
    \label{t:FirstVariations}
   Let $u \in \cE$. Let $\rho_u = \cT_g(u)\in H_g^1(\Omega)\cap H^2(\Omega)$ and
$\psi_u = \cL_u (uf'(\rho_u)) \in H_0^1(\Omega)\cap H^2(\Omega)$ be given as in 
Theorem~\ref{t:rhou} and 
Lemma~\ref{l:Lu}, respectively. 
If $v \in \cA_u$, then $J'[u][v]$ exists and is given by 
    \begin{equation}
    \label{Jpuv}
    J'[u][v] = \int_\Omega f(\rho_u) ( 1 - \psi_u) v \, dx.
    \end{equation}
If $ v \in \cA_u^+$, then the positive-side first variation $J'_+[u][v]$ exists
and is given by the right-hand side of \eqref{Jpuv}. In the case that $v \in \cA_u$, 
we have $v \in \cA_u^+$ and $J'_+[u][v] = J'[u][v].$
\end{theorem}

\begin{proof}
    We only prove \eqref{Jpuv} for the case $v \in \cA_u$ as the case that 
    $v \in \cA_u^+$ is similar (replacing $|t| \le t_v$ and $t\to 0$ by $0 \le t \le t_v^+$ and 
    $t\to 0^+$, respectively.) We divide our proof into three steps. 
 
{\it Step 1.} Since $v\in \cA_u$, there exists $t_v > 0$ such that $u+tv \in \cE$
if $|t| \le t_v.$
Denote $\rho = \rho_u$ and $\rho_t = \rho_{ u + t v} = \cT_g(u+tv)$ $(|t|\le t_v).$ We show that 
\begin{equation}
\label{rhot2rho}
\rho_t \to \rho \quad \mbox{in } H^1(\Omega) \quad \mbox{as } t \to 0.
\end{equation}
Note by Theorem~\ref{t:rhou} that 
$0 \le  \rho, \rho_t < g_{\rm max}$ 
in $\Omega$. 
Since $f$ increases in $[0, \infty)$, we  have
$0\le f(\rho), f(\rho_t) < f(g_{\rm max}) $ in $\Omega.$
Subtracting the weak formulations of $\rho_t$ and $\rho$ (cf.\ \eqref{LinftyH01}), we get
\begin{align}
\label{rhotrho}
&\int_\Omega \left[ \nabla (\rho_t-\rho) \cdot\nabla \phi  + \sigma (\rho_t-\rho)\phi + u (f(\rho_t)-f(\rho))\phi \right] dx 
\nonumber \\
&\qquad = -\int_{\Omega} tvf(\rho_t)\phi\, dx
\qquad \forall \phi \in H_0^1(\Omega). 
\end{align}
Setting $\phi=\rho_t-\rho \in H_0^1(\Omega) \cap H^2(\Omega)$
and noting that 
$u [ f(\rho_t) - f(\rho) ] (\rho_t - \rho) \ge 0$ a.e.\ in $\Omega$, we obtain for both the 
linear and nonlinear case that 
\begin{align*}
\min\{ 1, \sigma\} \| \rho_t - \rho \|^2_{H^1(\Omega)}
& \le 
\int_\Omega \left[ |\nabla (\rho_t-\rho)|^2  + \sigma (\rho_t-\rho)^2 \right] dx \\
&\le \left| -\int_{\Omega} tv f(\rho_t) (\rho_t - \rho) \, dx  \right| \\
&\le  |t|  f(g_{\rm max}) \,  \| v\|_{L^2(\Omega)} \| \rho_t - \rho \|_{L^2(\Omega)}. 
\end{align*}
Consequently, 
\begin{equation}
    \label{rtr}
\| \rho_t - \rho \|_{H^1(\Omega)} \le \frac{ f(g_{\rm max})}{\min\{ 1, \sigma\} }\| v \|_{L^2(\Omega)} | t| \qquad
\mbox{if }
|t|\le t_v. 
\end{equation}
This leads to \eqref{rhot2rho}.

{\it Step 2.} Let 
\[
\hat{\psi}_t = \frac{1}{t} \left( \rho_t - \rho \right) \in H_0^1(\Omega)\cap
L^\infty(\Omega) \qquad \mbox{for } 0 < |t| \le t_v. 
\]
Let $\hat{\psi}_{u, v} =\cL_u (- v f(\rho_u)) \in H_0^1(\Omega)\cap H^2(\Omega)$ be
given as in Lemma~\ref{l:Lu}.
We  show that 
\begin{equation}
    \label{psithatpu}
\hat{\psi}_t \to \hat{\psi}_{u, v}  \qquad \mbox{ weakly in } H^1(\Omega) \quad \mbox{as } t \to 0.
\end{equation}
It suffices to show that for any sequence $\{ t_n \}$ with $0 < | t_n | \le t_v$ and $t_n \to 0$
there exists a subsequence, not relabeled, such that $\hat{\psi}_{t_n} \to \hat{\psi}_{u, v}$ weakly
in $H_0^1(\Omega). $
Given $0 < |t_n | \le t_v$ with $t_n \to 0$. 
Since 
$\sup_{0 < |t| \le t_v} \| \hat{\psi}_t \|_{H^1(\Omega)} < \infty $ by \eqref{rtr}, 
there exists a  subsequence of $\{ t_n \}$, not relabeled, and $\psi \in H_0^1(\Omega)$ such that 
\begin{equation}
    \label{psitn2psi}
\hat{\psi}_{t_{n}} \to \psi \quad \mbox{weakly in } H^1(\Omega), 
\quad \mbox{strongly in } L^2(\Omega), \quad \mbox{and} \quad 
\mbox{a.e.\ in } \,   \Omega.
\end{equation} 
Moreover, passing to a further subsequence if necessary, we have by \eqref{rhot2rho} that 
\begin{equation}
    \label{rhotn}
    \rho_{t_n} \to \rho \qquad \mbox{strongly in } H^1(\Omega) \quad \mbox{and} \quad \mbox{ a.e. in }\,  \Omega. 
\end{equation}
We need to show that ${\psi} = \hat{\psi}_{u, v}$ in $H_0^1(\Omega).$

By \eqref{rhotrho}, we have 
for any $n \ge 1$ that
\begin{align}
\label{psitn}
\int_\Omega \left[ \nabla \hat{\psi}_{t_n} \cdot \nabla \phi  + \sigma \hat{\psi}_{t_n} \phi 
+ u \frac{f(\rho_{t_n})-f(\rho)}{t_n}\phi \right] dx & = -\int_{\Omega} vf(\rho_{t_n})\phi\, dx
\qquad \forall \phi \in H_0^1(\Omega).
\end{align}
Fix any $\phi \in H_0^1(\Omega)\cap L^\infty(\Omega)$. By \eqref{psitn2psi}, 
\[
\int_\Omega \left( \nabla \hat{\psi}_{t_n} \cdot \nabla \phi  + \sigma \hat{\psi}_{t_n} \phi \right) dx 
\to \int_\Omega \left( \nabla \psi   \cdot \nabla \phi  + \sigma {\psi}  \phi \right) dx. 
\]
For the linear case, $f(s) = s$ and $f'(s) = 1$ $(s\in \R)$. We then have by \eqref{psitn2psi}
and \eqref{rhotn} that 
\begin{align*}
    & \int_\Omega u \frac{f(\rho_{t_n})-f(\rho)}{t_n}\phi\, dx = 
    \int_\Omega u \hat{\psi}_{t_n} \phi \, dx \to 
    \int_\Omega u {\psi} \phi\, dx = \int_\Omega u f'(\rho) {\psi} \phi\, dx, 
    \\
    & -\int_\Omega v f(\rho_{t_n} ) \phi\, dx \to -\int_\Omega v f(\rho) \phi\, dx. 
\end{align*}
For the nonlinear case, $f(s) = s/(s+1)$ and $f'(s) = 1/(s+1)^2$ $(s > -1).$ We have 
\begin{align*}
 \int_{\Omega}u \frac{f(\rho_{t_{n}})-f(\rho)}{{t_{n}}}\phi \, dx
& 
=\int_{\Omega}u \frac{\rho_{t_{n}}-\rho}{{t_{n}}(\rho_{t_{n}}+1)(\rho+1)}\phi\, dx \\
&=\int_{\Omega} {u\phi} \frac{\hat{\psi}_{t_n}} {(\rho_{t_{n}}+1)(\rho+1)} \, dx 
\\
&= \int_\Omega \frac{u \phi}{(\rho_{t_n}+1) (\rho+1)} \left( \hat{\psi}_{t_n} - {\psi} \right)\, dx
\\
&\quad + \int_\Omega \frac{ u \phi \psi}{ (\rho + 1)} \left( \frac{1}{\rho_{t_n} + 1 }
-\frac{1}{ \rho+1} \right)dx + \int_\Omega \frac{u{\psi} \phi}{(\rho+1)^2}\, dx. 
\end{align*}
By \eqref{psitn2psi}, we get
\[
\left| \int_\Omega \frac{u \phi}{(\rho_{t_n}+1) (\rho+1)} \left( \hat{\psi}_{t_n} - \psi \right)\, dx 
\right|
\le \| u \phi \|_{L^2(\Omega)} \| \hat{\psi}_{t_n} - {\psi} \|_{L^2(\Omega)} \to 0.
\]
By H{\"o}lder's inequality, the embedding $H^1(\Omega) \hookrightarrow L^4(\Omega)$, and 
\eqref{rhotn}, we get for some constant $C > 0$ that 
\begin{align*}
\left| \int_\Omega \frac{ u \phi {\psi}}{ (\rho + 1)} \left( \frac{1}{\rho_{t_n} + 1 }
-\frac{1}{ \rho+1} \right)dx \right| 
& \le \| \phi\|_{L^\infty(\Omega)} \| u \|_{L^2(\Omega)} \| \psi \|_{L^4(\Omega)} 
\| \rho_{t_n} - \rho \|_{L^4(\Omega)} \\
& 
\le C \| \phi\|_{L^\infty(\Omega)} \| u \|_{L^2(\Omega)} \|\psi \|_{H^1(\Omega)}
\| \rho_{t_n} - \rho \|_{H^1(\Omega)}\\
&\to 0 \qquad \mbox{as } n \to \infty.
\end{align*}
Therefore, 
\[
\int_{\Omega}u \frac{f(\rho_{t_{n}})-f(\rho)}{{t_{n}}}\phi \, dx
\to \int_\Omega \frac{u{\psi} \phi}{(\rho+1)^2}\, dx = 
\int_\Omega u f'(\rho) {\psi} \phi\, dx. 
\]
Similarly, we have 
\[
-\int_\Omega v f(\rho_{t_n}) \phi \, dx 
\to -
\int_{\Omega}v f(\rho)\phi \, dx. 
\]
Combining all these limits, we obtain 
\begin{align*}
\int_\Omega \left[ \nabla \psi \cdot \nabla \phi  + \sigma \psi \phi 
+ u f'(\rho) \psi \phi\right] dx & = -\int_{\Omega} vf(\rho)\phi\, dx
\qquad \forall \phi \in H_0^1(\Omega) \cap L^\infty(\Omega)
\end{align*}
for both the linear and nonlinear case. 
Comparing this with the weak formulation for $\hat{\psi}_{u, v} = \cL_u(-v f(\rho_u))$, 
and recalling our notation $\rho = \rho_u$, 
we obtain  $\psi = \hat{\psi}_{u, v}$.
Hence, \eqref{psithatpu} holds true. 

{\it Step 3.}
We finally prove \eqref{Jpuv}. For the linear case, we have by \eqref{psithatpu} and
the compact embedding $H^1(\Omega)\hookrightarrow L^2(\Omega)$ that 
\begin{align}
\label{frhot1}
    \frac{f(\rho_t) - f(\rho)}{t} - f'(\rho) \hat{\psi}_{u, v} = \frac{\rho_t - \rho}{t} - \hat{\psi}_{u, v}
   = \hat{\psi}_t - \hat{\psi}_{u, v} \to 0
    \qquad \mbox{in } L^2(\Omega) \quad \mbox{as } t \to 0.
\end{align}
For the nonlinear case, we have by \eqref{rhot2rho}, \eqref{psithatpu}, and 
the compact embedding $H^1(\Omega)\hookrightarrow L^2(\Omega)$ that 
\begin{align}
\label{frhot2}
   & \frac{f(\rho_t)-f(\rho)}{t} - f'(\rho) \hat{\psi}_{u, v} 
  \nonumber  \\
   &\qquad = \left( \frac{\rho_t-\rho}{t} - \hat{\psi}_{u,v} \right) \frac{1}{(\rho_t + 1)(\rho+1)}
   + \hat{\psi}_{u, v} \frac{\rho-\rho_t}{(\rho+1)^2 (\rho_t+1)}
   \nonumber \\
   &\qquad 
 \to 0 \qquad \mbox{in } L^2(\Omega) \quad \mbox{as } t \to 0.
\end{align}
Consequently, by our notations
$ \rho = \rho_u$,  $\hat{\psi}_{u, v} = \cL_u ( - v f(\rho))$,
and $\psi_u = \cL_u ( u f'(\rho_u))$, and Lemma~\ref{l:Lu}, we have  
\begin{align*}
    J'[u][v]&=\lim_{t \to 0} \frac{J[u+tv]-J[u]}{t}\\
    &=\lim_{t\to 0} \frac{1}{t}\int_{\Omega} [(u+tv)f(\rho_t)-uf(\rho)] \,dx\\
    &=\lim_{t\to 0}\int_{\Omega} u\left( \frac{f(\rho_t)-f(\rho)}{t} \right) \, dx  
    +\lim_{t \to 0} \int_{\Omega} v(f(\rho_t)-f(\rho)) \, dx
    +\int_{\Omega} vf(\rho) \, dx\\
    &=\int_{\Omega} [uf'(\rho_u) \hat{\psi}_{u, v} + vf(\rho_u)] \, dx
    \\
    & = \int_\Omega \left[ uf'(\rho_u) \cL_u (-v f(\rho_u)) + v f(\rho_u) \right]dx
    \\
      & = \int_\Omega \left[ -v f(\rho_u ) \cL_u (uf'(\rho_u)) + v f(\rho_u) \right]dx
    \\
    & = \int_\Omega v f(\rho_u) (1 - \psi_u) \, dx. 
\end{align*}
This is \eqref{Jpuv}.
The proof is complete.
\end{proof}

Let $u \in \cE$. Let $v \in \cA_u$ be such that $u + t v \in \cE$
if $|t|\le t_v$ for some $t_v > 0.$ Fix any $t \in (-t_v, t_v).$  
Then $u + t v \in \cE$ and $v \in \cA_{u+tv}$. Hence, $J'[u+tv][v]$ exists
by Theorem~\ref{t:FirstVariations}. 
We define the second variation of $J:
\cE\to \R$ at $u$ in the direction of $v$ by 
\[ 
J''[u][v]=\lim_{t \to 0} \frac{J'[u+tv][v]-J'[u][v]}{t}, 
\] 
if the limit exists. Similarly, for any $v \in \cA_u^+$ and $t > 0$ small enough, 
$u + tv \in \cE$ and $v \in \cA_{u+tv}^+$. Moreover, 
$J'_+[u+tv][v]$ exists. We define the 
positive-side second variation 
\[ 
J''_+[u][v]=\lim_{t \to 0+} \frac{J'_+[u+tv][v]-J'_+[u][v]}{t}, 
\] 
if the limit exists. 

\begin{theorem}
\label{t:SecondVariations}
Let $u \in \cE$ and $ v \in L^2(\Omega)$. 
Let $\rho_u = \cT_g(u)\in H_g^1(\Omega) \cap H^2(\Omega)$,  
$\psi_u = \cL_u (uf'(\rho_u)) \in H_0^1(\Omega)\cap H^2(\Omega)$, and
 $\hat{\psi}_{u, v} =\cL_u (- v f(\rho_u)) \in H_0^1(\Omega)\cap H^2(\Omega)$ be given 
as in Theorem~\ref{t:rhou} and Lemma~\ref{l:Lu}, respectively. 
If $v \in \cA_u,$  then $J''[u][v]$ exists and is given by 
\begin{equation}
\label{J''uv}
J''[u][v] = \int_\Omega \left[ 2vf'(\rho_u)  (1-\psi_u)\hat{\psi}_{u, v}
+u f''(\rho_u) (1-\psi_u)\hat{\psi}_{u, v}^2\right] dx.
\end{equation}    
If $v \in \cA_u^+$, then $J''_+[u][v]$ exists and is given by the right-hand side of \eqref{J''uv}.
If $v \in \cA_u$ then $v \in \cA_u^+$ and $J''_+[u][v] = J''[u][v].$
\end{theorem}

\begin{proof}
   We only prove \eqref{J''uv} for the case $v \in \cA_u$ as the case that 
    $v \in \cA_u^+$ is similar (replacing $|t| \le t_v$ and $t\to 0$ by $0 \le t \le t_v^+$ and 
    $t\to 0^+$, respectively).
We prove the theorem in three steps.

{\it Step 1.}
Since $v \in \cA_u$, there exists $t_v > 0$ such that $u + t v \in \cE$ if $|t| \le t_v.$
Denote 
\[
\rho = \rho_u, \quad \rho_t = \rho_{u+tv}, \quad \psi = \psi_u, \quad 
\psi_t=\psi_{u+tv}, 
\qquad \mbox{with } \  |t| \le t_v.
\]
We show that 
\begin{equation}
\label{psit2psi}
\sup_{0 < |t| \le t_v} 
\left\| \frac{\psi_t - \psi }{t} \right\|_{H^1(\Omega)} < \infty
\qquad \mbox{and} \qquad 
\lim_{t\to 0} \| \psi_t - \psi \|_{H^1(\Omega)} = 0. 
\end{equation}

 Subtracting the weak formulations for $\psi_t$ and $\psi_u$ (cf.\ Lemma~\ref{l:Lu})
 and setting  the test function to be $\phi = \eta_t := \psi_t - \psi \in H_0^1(\Omega)\cap L^\infty(\Omega)$ in 
 the resulting equation, we have for $|t| \le t_v$ that 
\begin{align}
\label{etawf}
&\int_{\Omega} \left( |\nabla \eta_t |^2+\sigma  \eta_t^2 \right) dx 
\nonumber \\
& \qquad = \int_\Omega \left[ (u+tv) f'(\rho_t) - u f'(\rho) \right] \eta_t \, dx
-\int_\Omega \left[  (u+tv) f'(\rho_t)\psi_t- u f'(\rho)\psi \right] \eta_t  dx
\nonumber \\
& \qquad = \int_\Omega \left[ u  \left( f'(\rho_t) - f'(\rho) \right) \eta_t + 
t v f'(\rho_t) \eta_t \right] dx 
\nonumber \\
& \qquad \qquad 
- \int_\Omega \left[ u \left( f'(\rho_t)-f'(\rho)\right)\psi_t\eta_t 
+ uf'(\rho)\eta_t^2 + t v f'(\rho_t) \psi_t \eta_t \right] dx
\nonumber \\
&\qquad 
\le \int_\Omega \left[  u   \left( f'(\rho_t) - f'(\rho) \right)   \eta_t + 
t  v f'(\rho_t) \eta_t  \right] \, dx 
\nonumber \\
& \qquad \qquad 
- \int_\Omega \left[ u \left( f'(\rho_t)-f'(\rho)\right) \psi_t\eta_t 
 + t v f'(\rho_t) \psi_t \eta_t \right]  dx
\end{align}
For the linear case, $f(s) = s$ and $f'(s) = 1$ for any $s$. Hence, $f'(\rho_t) - f'(\rho) = 0$ in $\Omega.$ 
For the nonlinear case, $f(s) = s/(s+1) $ $(s \ge 0)$. Thus, $0< f'(s) = 1/(s+1)^2$ $(s \ge 0)$. Hence, 
\[
\left| f'(s) - f'(\hat{s}) \right| = \frac{| 2+ s  + \hat{s}  | \, | s - \hat{s} |}{(1 + s)^2(1+\hat{s})^2}
\le \frac{ (|1 + s| + |1 + \hat{s}| ) | s - \hat{s}|}{ (1+s)^2 (1+\hat{s})^2}
\le 2 |s - \hat{s}| \qquad \forall s, \hat{s} \ge 0. 
\]
Moreover, by Lemma~\ref{l:Lu}, we have $0 \le \psi_t \le 1$ in $\Omega$,
and hence $|\eta_t| \le 2$ in $\Omega.$ Consequently, it follows from 
 \eqref{etawf}, the embedding $H^1(\Omega)\hookrightarrow L^6(\Omega),$  H{\"o}lder's inequality, 
and \eqref{rtr} that 
\begin{align*}
\min\{ 1, \sigma\} \| \eta_t \|^2_{H^1(\Omega)}
& \le \int_\Omega \left( |\nabla \eta_t|^2 + \sigma \eta_t^2 \right)\, dx 
\\
&\le 4 \int_\Omega u | \rho_t - \rho | \, |\eta_t | \, dx + 2 |t| \int_\Omega  |v | \, |\eta_t|\, dx\\
&\le 4\|u\|_{L^2(\Omega)}\|\rho_t-\rho\|_{L^6(\Omega)}\|\eta_t\|_{L^3(\Omega)}
+2|t|\|v\|_{L^2(\Omega)}\|\eta_t\|_{L^2(\Omega)} \\
&\le C_1(u, v) \left( \| \rho_t - \rho \|_{H^1(\Omega)} \| \eta_t \|_{H^1(\Omega)}
+ | t| \|\eta_t\|_{H^1(\Omega)} \right) \\
& \le C_2(u, v) |t| \| \eta_t \|_{H^1(\Omega)} \qquad \mbox{if } |t|\le t_v, 
\end{align*}
where $C_i(u, v)> 0$ $(i=1, 2)$ are constants independent of $t.$ 
This leads to the inequality and hence the limit in \eqref{psit2psi}.

{\it Step 2.} Let \[ \xi_t = \frac{1}{t} \left( \psi_t - \psi \right) \in H_0^1(\Omega)\cap L^\infty(\Omega) \qquad \mbox{for } 0 < |t| \le t_v. \] 
We  show that 
\begin{equation} 
\label{xit2xiuv} 
\xi_t \to \xi_{u, v} \quad \mbox{ weakly in } H^1(\Omega) \quad \mbox{as } t \to 0, 
\end{equation} 
where 
\begin{equation}
\label{xiuv}
\xi_{u, v} =\cL_u \left(vf'(\rho)(1-\psi_u)+uf''(\rho)(1-\psi_u)\hat{\psi}_{u, v} \right)\in H_0^1(\Omega)\cap H^2(\Omega).
\end{equation}
It suffices to show that for any sequence $\{ t_n \}$ with $0 < |t_n| \le t_v$ and $t_n \to 0$
there exists a subsequence, not relabeled, such that $\xi_{t_n} \to \xi_{u,v}$ weakly
in $H_0^1(\Omega). $
Let $0 < |t_n | \le t_v$ $(n = 1, 2, \dots)$ be such that $t_n \to 0$. 
Since $ \sup_{0 < |t| \le t_v} \| \xi_t \|_{H^1(\Omega)} < \infty$ by \eqref{psit2psi}, 
there exists a  subsequence of $\{ t_n \}$, not relabeled, and $\hat{\xi} \in H_0^1(\Omega)$ such that 
\begin{equation}  
\label{xitn2xi}
\xi_{t_{n}} \to \hat{\xi} \quad \mbox{weakly in } H^1(\Omega), 
\quad \mbox{strongly in } L^2(\Omega), \quad \mbox{and} \quad 
\mbox{a.e.\ in }\,  \Omega.
\end{equation} 
Moreover, passing to a further subsequence if necessary, we have by \eqref{psithatpu}, \eqref{psitn2psi},  and 
the compact embedding $H^1(\Omega) \hookrightarrow L^3(\Omega)$ with $\rho = \rho_u$ that
\begin{equation}
    \label{psitn2psiuv}
    \frac{\rho_{t_n} - \rho}{t_n} \to \hat{\psi}_{u, v} \quad \mbox{weakly in } H^1(\Omega), 
\quad \mbox{strongly in } L^3(\Omega), \quad \mbox{and} \quad 
\mbox{a.e.\ in } \,   \Omega.
\end{equation}
We need to show that $\hat{\xi}=\xi_{u, v}$ in $H_0^1(\Omega).$

It follows from the weak forms for $\psi_t$ and $\psi$, respectively, that for any $n \ge 1$
\begin{align}
\label{xitn}
&\int_{\Omega} \left[\nabla \xi_{t_n} \cdot \nabla \phi +\sigma  \xi_{t_n} \phi 
+u\left(\frac{f'(\rho_{t_n})-f'(\rho)}{t_n}\right)\psi_{t_n} \phi
+  uf'(\rho) \xi_{t_n} \phi\,
+v f'(\rho_{t_n})\psi_{t_n}\phi \right]\,dx \nonumber \\
&\qquad = \int_{\Omega}\left[\frac{u(f'(\rho_{t_n})-f'(\rho))}{t_n}\phi 
+vf'(\rho_{t_n})\phi \right]\, dx \qquad \forall \phi \in H_0^1(\Omega).  
\end{align}
Fix $\phi \in H_0^1(\Omega)\cap L^\infty(\Omega)$. By \eqref{xitn2xi}
and the fact that $0 \le f'(\rho) \le 1$ in $\Omega$, we have
\begin{align}
\label{xin1}
&\int_\Omega \left[ \nabla \xi_{t_n} \cdot \nabla \phi  + \sigma \xi_{t_n} \phi 
+uf'(\rho) \xi_{t_n} \phi \right] dx 
\to \int_\Omega \left[ \nabla \hat{\xi} \cdot \nabla \phi  + \sigma \hat{\xi}  \phi 
+ uf'(\rho) \hat{\xi} \phi \right] dx. 
\end{align}
The third part of the left integral in \eqref{xitn} vanishes for the linear case where $f'(s) = 1$ for all $s$. 
Consider the nonlinear case.  For any $t$ with $|t|\le t_v$, set
\[
\alpha (\rho_t, \rho) = \frac{1}{ (1+\rho_t)(1+\rho)}\left( \frac{1}{1  + \rho_t} + \frac{1}{1+\rho} \right).
\]
We verify that $0 < \alpha(\rho_t, \rho)\le 2$ in $\Omega$, and 
\begin{align}
\label{fpfp}
&f'(\rho_t) - f'(\rho) = -(\rho_t - \rho) \alpha (\rho_t, \rho), 
\\
&f''(\rho) = - \frac{2}{(\rho+1)^3 }= - \alpha (\rho, \rho).
\end{align}
Therefore, denoting $\hat{\psi} = \hat{\psi}_{u, v}, $ we have 
\begin{align*}
&\int_{\Omega} u\left(\frac{f'(\rho_{t_n})-f'(\rho)}{t_n}\right)\psi_{t_n} \phi \,dx \\
& \qquad = -\int_\Omega \left( \frac{\rho_{t_n} - \rho}{t_n} \right) \alpha (\rho_{t_n}, \rho)
\left( \psi_{t_n} - \psi \right) u \phi \,dx - 
\int_\Omega \left( \frac{\rho_{t_n} - \rho}{t_n} - \hat{\psi} \right)
\alpha (\rho_{t_n}, \rho) \psi u \phi \, dx\\
& \qquad \qquad - \int_\Omega \hat{\psi} \left( \alpha(\rho_{t_n}, \rho) - \alpha (\rho, \rho) \right)
\psi  u \phi \, dx + \int_\Omega  \hat{\psi} f''(\rho) \psi u \phi\, dx\\
&\qquad:= -a_n - b_n - c_n + d. 
\end{align*}
Note that $u \phi \in L^2(\Omega)$ and $0\le \alpha (\rho_{t_n}, \rho) \le 2$ in $\Omega$ 
for all $n$. By \eqref{psitn2psiuv}, $(\rho_{t_n} - \rho)/t_n \to \hat{\psi}$ 
strongly in $L^3(\Omega).$ By \eqref{psit2psi} and the embedding
$H^1(\Omega) \hookrightarrow L^6(\Omega)$, $\psi_{t_n} \to \psi$ strongly in $L^6(\Omega).$
An application of H{\"o}lder's inequality then leads to 
\[
a_n = \int_\Omega \left( \frac{\rho_{t_n} - \rho}{t_n} \right) \alpha (\rho_{t_n}, \rho)
\left( \psi_{t_n} - \psi \right) u \phi \, dx \to 0 \qquad \mbox{as } n \to \infty.
\]
Similarly, since $u\psi \phi \in L^2(\Omega)$,  we have 
\[
b_n = \int_\Omega \left( \frac{\rho_{t_n} - \rho}{t_n} - \hat{\psi} \right)
\alpha (\rho_{t_n}, \rho) \psi u \phi \, dx \to 0 \qquad \mbox{as } n \to \infty.
\]
Since $\hat{\psi} \psi u \phi \in L^2(\Omega)$, it follows from \eqref{rhotn} and an application
of the Lebesgue Dominated Convergence Theorem that 
\[
c_n = \int_\Omega \hat{\psi} \left( \alpha(\rho_{t_n}, \rho) - \alpha (\rho, \rho) \right)
\psi  u \phi \, dx \to 0 \qquad \mbox{as } n \to \infty.
\]
Consequently, we have the convergence of the third part of the integral on the left of \eqref{xitn}
\begin{equation}
\label{xin2}
    \int_{\Omega} u\left(\frac{f'(\rho_{t_n})-f'(\rho)}{t_n}\right)\psi_{t_n} \phi \,dx 
    \to d=\int_\Omega  \hat{\psi} f''(\rho) \psi u \phi\, dx \qquad \mbox{as } n \to \infty.
\end{equation}
This and similar arguments also lead to the convergence of the last part of the integral on the left-hand side of \eqref{xitn}
\begin{align}
\label{xin3}
\int_\Omega v f'(\rho_{t_n} ) \psi_{t_n} \phi \, dx &= 
t_n \int_\Omega v \left( \frac{f'(\rho_{t_n}) - f'(\rho)}{t_n} \right) \psi_{t_n} \phi \, dx
\nonumber \\
& \qquad 
+ \int_\Omega v f'(\rho) (\psi_{t_n} - \psi) \phi \, dx 
 + \int_\Omega v f'(\rho) \psi \phi \, dx
 \nonumber \\
& \to \int_\Omega v f'(\rho) \psi \phi \, dx \qquad \mbox{as } n \to \infty.
\end{align}
Applying similar arguments, we obtain the convergence of the right-hand side of \eqref{xitn}
\begin{align}
\label{xin4}
\int_{\Omega} \left[ \frac{u(f'(\rho_{t_n})-f'(\rho))}{t_n}\phi 
+vf'(\rho_{t_n})\phi \right]\, dx \to
\int_{\Omega} \left[ uf''(\rho)\hat{\psi}\phi + vf'(\rho)\phi
\right] \, dx.
\end{align}
It now follows from \eqref{xitn}, \eqref{xin1}, \eqref{xin2},  \eqref{xin3}, 
and \eqref{xin4}, and rearranging
terms that
\begin{align*}
&\int_{\Omega} \left[\nabla \hat{\xi} \cdot \nabla \phi +\sigma  \hat{\xi} \phi + 
uf'(\rho)\hat{\xi}\phi \right]\,dx 
\nonumber \\
& \qquad  =\int_{\Omega} \left[ vf'(\rho_u) (1 - \psi) + u f''(\rho) \hat{\psi} ( 1 - \psi) \right]\phi\, dx
\qquad \forall \phi \in H_0^1(\Omega) \cap L^\infty(\Omega). 
\end{align*}
Since  all $\psi, \hat{\psi}, f'(\rho_u), f''(\rho) \in L^\infty(\Omega)$ and 
$\hat{\xi} \in H^1(\Omega) \hookrightarrow L^4(\Omega)$, we can replace $\phi \in H_0^1(\Omega)\cap L^\infty(\Omega)$
by $\phi\in H_0^1(\Omega) \hookrightarrow L^4(\Omega).$
Hence, with our notation $\rho = \rho_u$, $\psi = \psi_u $, and $\hat{\psi} = \hat{\psi}_{u, v}$, and 
with the definition of $\cL_u$ (cf.\ Lemma~\ref{l:Lu}) and 
$\xi_{u, v}$ (cf.\ \eqref{xiuv}), we have $\hat{\xi} = \xi_{u, v}$ as desired. 

{\it Step 3.} 
It follows from \eqref{frhot1}, \eqref{frhot2}, \eqref{psit2psi}, and \eqref{xit2xiuv} that
\begin{align*}
    J''[u][v]&=\lim_{t \to 0} \frac{J'[u+tv][v]-J'[u][v]}{t}\\
    &=\lim_{t\to 0} \frac{1}{t}\int_{\Omega} [f(\rho_t)(1-\psi_t)-f(\rho)(1-\psi)]v \,dx\\
    &=\lim_{t\to 0}\int_{\Omega} \frac{f(\rho_t)-f(\rho)}{t}(1-\psi_t)v \, dx 
    -\lim_{t \to 0} \int_{\Omega} \frac{\psi_t-\psi}{t}f(\rho)v \, dx\\
    &=\int_{\Omega} f'(\rho) \hat{\psi}(1-\psi)v \, dx -\int_{\Omega} \xi_{u,v} vf(\rho) \, dx  \\
    &= \int_{\Omega} f'(\rho) \hat{\psi}(1-\psi)v \, dx 
    -\int_{\Omega} \cL_u \left(vf'(\rho)(1-\psi)+uf''(\rho)(1-\psi)\hat{\psi} \right) vf(\rho) \, dx  \\
    &= \int_{\Omega} f'(\rho) \hat{\psi} (1-\psi)v \, dx 
    +\int_{\Omega} \left(vf'(\rho)(1-\psi)+uf''(\rho)(1-\psi)\hat{\psi} \right) \cL_u (-vf(\rho)) \, dx  \\
    &= \int_{\Omega} f'(\rho) \hat{\psi} (1-\psi)v \, dx +\int_{\Omega} \left(vf'(\rho)(1-\psi)
    +uf''(\rho)(1-\psi)\hat{\psi} \right) \hat{\psi}\, dx \\
    & =\int_\Omega \left[ 2vf'(\rho)  (1-\psi)\hat{\psi}
+u f''(\rho) (1-\psi)\hat{\psi}^2\right] dx.
\end{align*}
This is \eqref{J''uv} with our notation $\rho = \rho_u$, $\psi = \psi_u$, and $\hat{\psi} = \hat{\psi}_{u, v}.$
The proof is complete.
\end{proof}

\begin{theorem}
\label{t:alphav}
Let $u \in \cE$. 
\begin{compactenum}
    \item[{\rm (1)}]
Let $v \in \cA_u$ be such that $u + t v \in \cE$ if $|t| \le t_v$ for 
some $t_v > 0.$ Define $\alpha_v(t) = J[u+tv]$ $(|t| < t_v).$ Then, $\alpha_v \in 
C^2((-t_v, t_v))$. Moreover, 
$\alpha_v'(t) = J'[u+tv][v] $ and $\alpha_v''(t) = J''[u+tv][v]$  for any $t \in (-t_v, t_v)$, and
\begin{equation}
\label{Jutv}
J[u+tv] = J[u] + J'[u][v]t + \frac12 J''[u][v] t^2 + o(t^2) \qquad \mbox{as } t \to 0.
\end{equation}
\item[{\rm (2)}] Let $v \in \cA_u^+$ be such that $u + t v \in \cE$ if $0 \le t \le t_v^+$ for
some $t_v^+ > 0$. Define $\alpha_v^+(t)= J[u+tv]$ for $t \in [0, t_v^+)$. Then, 
$\alpha_v^+\in C^2([0, t_v^+))$. Moreover, 
${\alpha_v^+}'(t) = J'_+[u+tv][v]$ and ${\alpha_v^+}''(t) = J''_+[u+tv][v]$ for any 
$t \in [0, t_v^+),$ and 
\begin{equation}
\label{Jutv+}
J[u+tv] = J[u] + J'_+[u][v]t + \frac12 J''_+[u][v] t^2 + o(t^2) \qquad \mbox{as } t \to 0^+.
\end{equation}
If $v \in \cA_u$, then $v \in \cA_u^+$ and $\alpha_v^+ (t) = \alpha_v(t)$ if
$t \in [0, t_v)$, where $t_v$ is defined in Part (1).
\end{compactenum}
\end{theorem}

\begin{proof}
We only prove Part (1), as the proof of Part (2) is similar.
Fix any $t_0 \in (-t_v, t_v)$. Then $u+t_0v \in \cE$ and $v \in \cA_{u+t_0v}$ as $u+t_0v + t v \in \cE$ 
if $|t|< t_v - |t_0|. $ 
Thus, by Theorem~\ref{t:FirstVariations}, $J'[u+t_0 v][v]$ exists, and 
\begin{align*}
\alpha_v'(t_0)  &= \lim_{\tau \to 0} \frac{\alpha_v(t_0+\tau) - \alpha_v(t_0)}{\tau} 
\\
& =\lim_{\tau \to 0} \frac{J[u + (t_0+\tau)v] - J[u+t_0v ]}{\tau}
\\
&=J'[u+t_0v][v].
\end{align*}
Similarly, $J''[u+t_0v][v]$ exists and $\alpha_v''(t_0) = J''[u+t_0v][v].$ 

Since $\alpha_v''$ exists on $(-t_v, t_v)$, $\alpha_v'$ is differentiable and hence continuous, therefore $\alpha_v \in C^1((-t_v, t_v)).$ For $|t-t_0|<\delta$ where $\delta:=t_v-|t_0|>0$, denote
\begin{equation}
\label{urpp}
u_t  = u+tv, \quad \rho_t = \rho_{u_t} = \cT_g(u_t), \quad \psi_t = \cL_{u_t} (u_t f'(\rho_t)), 
\quad \hat{\psi}_t = \hat{\psi}_{u_t, v} = \cL_{u_t} (-v f(\rho_t)). 
\end{equation}
By Theorem~\ref{t:rhou} and Lemma~\ref{l:Lu}, we have 
$\rho_t \in H_g^1(\Omega) \cap H^2(\Omega)$ and $\psi_t, \hat{\psi}_t \in H_0^1(\Omega)\cap H^2(\Omega).$ Moreover, it follows from \eqref{rhot2rho} and \eqref{psit2psi}, with $u+t_0 v$ replacing
$u$, that $\rho_t \to \rho_{t_0}$ 
and $ \psi_t \to \psi_{t_0}$ in $H^1(\Omega)$ as $t \to t_0$. Similarly, 
with the same argument in proving $ \psi_t \to \psi_{t_0}$ in $H^1(\Omega)$, we obtain
$\hat{\psi}_t \to \hat{\psi}_{t_0}$ in $ H^1(\Omega)$ as $t \to t_0.$ 
By Theorem~\ref{t:SecondVariations}, 
\begin{equation}
\label{app}
\alpha_v''(t) = J''[u_t][v]=\int_{\Omega} 
\left( 2 v f'(\rho_t)(1-\psi_t) \hat{\psi}_t
+u_t f''(\rho_t)(1-\psi_t)\hat{\psi}_t^2 \right) \, dx =: A_t + B_t. 
\end{equation}
Note that $f'$ is Lipschitz on $[0,g_{\rm max}]$. Thus, by the embedding $H^1(\Omega) \hookrightarrow L^4(\Omega)$, we have 
$f'(\rho_t)\to f'(\rho_{t_0})$ in $L^4(\Omega).$ Consequently, by 
H{\"o}lder's inequality and the fact $|f'(\rho_t)|\le 1$ and $0 \le \psi_t \le 1$ in $\Omega$, we obtain that 
\begin{align}
\label{AtAt}
| A_t - A_{t_0} | & = \left| \int_\Omega \left[ 2 v f'(\rho_t)(1-\psi_t) \hat{\psi}_t
- 2 v f'(\rho_{t_0} )(1-\psi_{t_0} ) \hat{\psi}_{t_0} \right] dx  \right|
\nonumber \\
& \le \int_\Omega \left| 2 v \left[ f'(\rho_t) - f'(\rho_{t_0} ) \right] (1 - \psi_{t})
\hat{\psi}_t \right| \, dx + \int_\Omega \left| 2 v f'(\rho_{t_0}) 
( \psi_{t_0} - \psi_{t} ) \hat{\psi}_t \right|\, dx \nonumber
\\
&\qquad + \int_\Omega \left|  2 v f'(\rho_{t_0})(1 - \psi_{t_0} ) ( \hat{\psi}_t - \hat{\psi}_{t_0} )\right|\, dx 
\nonumber \\
&\le 
2 \|v\|_{L^2(\Omega)} \left(  \|f'(\rho_t)-f'(\rho_{t_0})\|_{L^4(\Omega)} \|\hat{\psi}_t\|_{L^4(\Omega)} 
+  \|\psi_t - \psi_{t_0}\|_{L^4(\Omega)} \| \hat{\psi}_{t} \|_{L^4(\Omega)}
\right. 
\nonumber \\
&\qquad 
\left. 
+ \| \hat{\psi}_t - \hat{\psi}_{t_0}\|_{L^2(\Omega)} \right)
\nonumber \\
& \to 0 \qquad \mbox{as } t \to t_0. 
\end{align}
The second term $B_t$ vanishes for the linear case where $f''(s) = 0$ for any $s \in \R.$ For the nonlinear case, we have
by the fact that $| f''(s) | \le 2$ and $|f'''(s)|\le 6$ for any $s \ge 0$, 
$|f''(s) - f''(s')| \le 6 | s - s'|$ for any $s, s'\ge 0$, and $0 \le \psi_t, \psi_{t_0} \le 1$ in $\Omega$, 
\begin{align*}
    &\left| u_t f''(\rho_t)(1-\psi_t)\hat{\psi}_t^2 
    - u_{t_0}  f''(\rho_{t_0} )(1-\psi_{t_0} )\hat{\psi}_{t_0} ^2 \right|\\
    &\qquad \le \left| (u_t-u_{t_0}) f''(\rho_t) (1 - \psi_t )\hat{\psi}_t^2 \right|
    + \left| u_{t_0} \left[ f''(\rho_t)-f''(\rho_{t_0}) \right] (1-\psi_t) \hat{\psi}_t^2 \right| \\
    & \qquad \qquad 
    + \left| u_{t_0} f''(\rho_{t_0})  (\psi_{t_0}-\psi_t) \hat{\psi}_t^2 \right|
    + \left| u_{t_0} f''(\rho_{t_0})  (1- \psi_{t_0}) (\hat{\psi}_t - \hat{\psi}_{t_0})
    (\hat{\psi}_t+\hat{\psi}_{t_0}) \right|
    \\
    & \qquad \le 2|t-t_0| |v| \hat{\psi}_t^2 + 6 u_{t_0} \left| \rho_t-\rho_{t_0} \right| \hat{\psi}_t^2 
    + 2u_{t_0} | \psi_t - \psi_{t_0} |\hat{\psi}_t^2 + 2u_{t_0} | \hat{\psi}_t - \hat{\psi}_{t_0} |
    \left( |\hat{\psi}_t| + |\hat{\psi}_{t_0} | \right). 
\end{align*}
Note that $ \| \hat{\psi}_t^2 \|_{L^2(\Omega)} = \| \hat{\psi}_t \|^2_{L^4(\Omega)}$
and $\| \hat{\psi}_t^2 \|_{L^3(\Omega)} = \| \hat{\psi}_t \|^2_{L^6(\Omega)}$. 
Consequently, we have by the Sobolev embedding $H^1(\Omega)\hookrightarrow L^6(\Omega)$, H{\"o}lder's inequality, 
and the convergence $\rho_t \to \rho_{t_0}$ (cf.\ \eqref{rhot2rho}), $ \psi_t \to \psi_{t_0}$ 
(cf.\ \eqref{psit2psi}),  and $\hat{\psi}_t\to 
\hat{\psi}_{t_0}$ (cf.\ \eqref{psit2psi} and \eqref{urpp}) in $H^1(\Omega)$ as $ t \to t_0$ that 
\begin{align*}
    | B_t - B_{t_0} | &  = 
    \left| \int_\Omega \left[ u_t f''(\rho_t)(1-\psi_t)\hat{\psi}_t^2 
    - u_{t_0} f''(\rho_{t_0} )(1-\psi_{t_0} )\hat{\psi}_{t_0}^2 \right] dx\right|
    \\
    & \le 2|t-t_0| \| v\|_{L^2(\Omega)} \| \hat{\psi}_t\|^2_{L^4(\Omega)}
    + 6  \| u_{t_0} \|_{L^2(\Omega)} \| \rho_t - \rho_{t_0} \|_{L^6(\Omega)}
    \| \hat{\psi}_t \|_{L^6(\Omega)}^2 \\
    &\qquad +2 \| u_{t_0} \|_{L^2(\Omega)} \| \psi_t - \psi_{t_0} \|_{L^6(\Omega)}
    \| \hat{\psi}_t\|^2_{L^6(\Omega)} \\
    & \qquad + 2\| u_{t_0} \|_{L^2(\Omega)}
    \| \hat{\psi}_t - \hat{\psi}_{t_0} \|_{L^4(\Omega)} \left( \| \hat{\psi}_{t} \|_{L^4(\Omega)} 
    + \| \hat{\psi}_{t_0} \|_{L^4(\Omega)} \right) \\
    & \to 0 \qquad \mbox{as } t \to t_0. 
\end{align*}
This, together with \eqref{app} and \eqref{AtAt}, implies that
$\alpha_v''(t) \to \alpha_v''(t_0)$ as $t \to t_0$. 
Thus,  $\alpha''_v$ is continuous on $(-t_v, t_v).$
Finally, we obtain \eqref{Jutv} by Taylor's expansion.
\end{proof}

\section{Critical Points and Local Maximizers}
\label{s:CriticalPoints}

\begin{definition}
    \label{d:CriticalPoint} We say that 
    $u \in \cE$ is a critical point of $J: \cE\to [0, \infty)$ if 
    $J'[u][v] = 0$ for all $v \in \cA_u.$

    We say that $u \in \cE$ is a local maximizer of $J: \cE \to [0, \infty)$, 
    if there exists a $\delta> 0$ such that 
    $J[u] \ge J[w] $ for any $w \in \cE $ with $ \|u-w\|_{L^2(\Omega)} < \delta.$
\end{definition}

\begin{lemma}
\label{l:Jpu}
If $u \in \cE$ is a local maximizer of $J: \cE \to \R$, then it is a critical point of $J: \cE\to \R$. Moreover, 
$J'_+[u][v ] \le 0$ for any $v \in \cA_u^+.$
\end{lemma}

\begin{proof}
Let $u \in \cE$ be a local maximizer of $J: \cE \to \R.$ 
By definition, there  exists $\delta >0$ such that $J[u] \ge J[w]$ for any 
$w \in \cE $ with $ \|u-w\|_{L^2(\Omega)} < \delta.$ For any $v\in\cA_u$, 
there exists $t_v>0$ such that
$u+tv\in\cE$ for all $|t|\le t_v$. Set 
$t_0:=\min\{ t_v, \delta/(1 + \| v \|_{L^2(\Omega)})\}>0$. 
Then for every $|t|\le t_0$ we have 
$ u + t v \in \cE$ and $\|u-(u+tv)\|_{L^2(\Omega)}=|t|\|v\|_{L^2(\Omega)}<\delta.$
Hence $J[u]\ge J[u+tv] \,\mbox{for all } |t|\le t_0$. 
By Theorem $\ref{t:alphav}$, $\alpha_v (t)=J[u+tv]\in C^1(-t_v,t_v)$. 
Moreover, $\alpha_v$ attains a local maximum at $t=0$. Thus, $\alpha_v'(0)=J'[u][v]=0$.
Hence, $u$ is a critical point of $J: \cE \to \R.$

For any $v \in \cA_u^+$, there exists $t_v^+ > 0$ such that $u + t v \in \cE$ if 
$0 \le t \le t_v^+ $. Let $t_0^+ = \min\{ t_v^+, \delta/(1 + \| v \|_{L^2(\Omega)})\}> 0$. 
Then $u + t v \in \cE$ and $ \| u + t v - u \|_{L^2(\Omega)} < \delta,$ 
if $ 0 \le t \le t_0^+.$ Hence 
$J[u+tv] \le J[u]$ for any $t \in [0, t_0^+]$. By the same reasoning 
as above and Theorem \ref{t:alphav}, we have that the right derivative ${\alpha_v^+}'(0)=J'_+[u][v]\le 0$  for all $v \in \cA_u^+ $. 
\end{proof}

The following is our main result in this section: 

\begin{theorem}
\label{t:Jmax} The functional $J: \cE \to \R$ admits no local maximizers. 
\end{theorem}

To prove the theorem, we need to study the conditions that, for a given $u \in \cE,$ 
$J'[u][v] = 0$ for any $v \in \cA_u$ and
$J'_+[u][v] \le 0$ for any $v \in \cA_u^+$.
Let us first  characterize
the set $\cA_u$ defined in \eqref{cAu}. Let $u\in \cE.$ We define 
\begin{align*}
&\hat{\cA}_u = \left\{ v \in L^2(\Omega): 
\int_\Omega v\, dx = 0 
\mbox{ and } |v| \le C_v u \mbox{ a.e.\ in }\, \Omega \mbox{ for some } C_v > 0 \right\}, \\
& \tilde{\cA}_u = \left\{ v \in L^2(\Omega): 
\int_\Omega v\, dx = 0 
\mbox{ and } v = u \varphi \mbox{ a.e.\ in }\, \Omega \mbox{ for some } \varphi \in L^\infty(\Omega)  \right\},
\\
&
\cB_u  = \left\{ w \in L^2(\Omega): 
\int_\Omega w\, dx = 0
\mbox{ and } w=0 \mbox{ a.e.\ in }\, \{ u =0\}\right\}, 
\end{align*}
where $\{ u = 0 \} = \{ x \in \Omega: u(x) = 0 \}.$
These definitions do not depend on the representative of $u$ in $L^1(\Omega).$
Moreover, it is clear that all $\hat{\cA}_u$, $\tilde{\cA}_u$, and 
$\cB_u$ are subspaces of $L^2(\Omega)$.

\begin{lemma}
    \label{l:AuBu} 
Let $u \in \cE.$ Then, $\cA_u = \hat{\cA}_u=\tilde{\cA}_u\subseteq \cB_u ={\overline{\cA_u}^{L^2(\Omega)}},$ 
where ${\overline{{\cA}_u}}^{L^2(\Omega)}$ denotes  the closure of $\cA_u$ in $L^2(\Omega)$. 
\end{lemma}

\begin{proof}
    Suppose $v \in \cA_u$. Then there exists $t_v>0$ such that $u+tv \in \cE$ 
    if $|t|\le t_v$. By the definition of $\cE$, we must have 
    \[
    \displaystyle{\int_\Omega v \, dx=0}.
    \]
    We also have $tv = u+tv \ge 0$ a.e.\ in the set $\{u=0\}$ if $|t|\le t_v$. By considering
    $t > 0$ and $t < 0$, we find that
    $v = 0$ a.e.\ in $\{ u = 0 \}$. Moreover, we have $u+t_v v\ge 0$ and $u-t_vv\ge 0$ a.e.\
    in the set $\{u>0\}$, implying that $ |v | \le (1/t_v) u$ a.e.\ in $\{ u > 0 \}. $
Hence, $|v| \le C_v u$ a.e.\ in $\Omega$ with $C_v=1/t_v$ and hence $v \in \hat{\cA}_u.$
Conversely, suppose $v \in \hat{\cA}_u$. Then
$|v| \le C_v u$ a.e.\ in $\Omega$ for some $C_v > 0$. Set $t_v=1/{C_v}$.
Then, for $|t| \le t_v,$ we have $tv \ge -t_v |v| \ge - t_v C_v u=-u$. 
Thus, $u+tv\ge 0$ a.e.\ in $\Omega$, and hence $u + t v \in \cE,$
for $|t| \le t_v$, i.e. $v \in \cA_u$. Hence, $\cA_u = \hat{\cA}_u.$

Clearly, $\tilde{\cA}_u \subseteq \hat{\cA}_u.$ Let $v \in \hat{\cA}_u.$ Thus, 
$ |v| \le C_v u $ a.e.\ in $\Omega$ for some $C_v > 0. $ Define $\varphi$ by
$\varphi(x) = 0$ if $u(x) = 0$ and $\varphi(x) = v(x)/u(x)$ if $u(x) \ne 0.$ Then, 
$\varphi \in L^\infty(\Omega)$ and $v = u \varphi$ a.e.\ in $\Omega$. Hence, $v \in \tilde{\cA}_u$
and  $\hat{\cA}_u \subseteq \tilde{\cA}_u.$ Thus, $\hat{\cA}_u = \tilde{\cA}_u.$

If $v \in \hat{\cA}_u$, then $v = 0$ a.e.\ in $\{ u = 0 \}$. Hence, $v \in \cB_u.$ 
Thus, $\cA_u = \hat{\cA}_u = \tilde{\cA}_u \subseteq \cB_u.$

Let $w \in \cB_u$
and define $v_n=\max\{-nu,\min\{nu,w\}\}$ for each integer $n \ge 1$. 
Thus, $|v_n| \le |w|$ a.e.\ in $\Omega$ and $v_n = 0$ a.e.\ in $\{ u = 0 \}.$ 
Since $w \in \cB_u$
we infer that $v_n \in L^2(\Omega)$. 
Moreover, 
$v_n \to w \chi_{\{u>0\}}=w$ a.e.\ in $\Omega$. By the Lebesgue
Dominated Convergence Theorem, $v_n \to w $ in $L^2(\Omega)$. 
Since $u \in \cE,$ we have $| \{ u > 0 \} | > 0.$ Define 
\[
w_n=v_n-c_n u \qquad \mbox{ where } c_n= \frac{1}{M_{\rm e}}\int_{\Omega} v_n\,dx
\qquad(n=1, 2, \dots). 
\]
Then, for each $n \ge 1,$ $w_n \in L^2(\Omega)$  and 
\[
\int_\Omega w_n \, dx = 0. 
\]
Moreover, $|w_n| \le |v_n| +|c_n|u \le ( n+ |c_n|) u$ a.e.\ in $\Omega.$  
Hence, each $w_n \in \hat{\cA}_u = \cA_u.$
Since $v_n \to w$ in $L^2(\Omega)$ and $w\in \cB_u$, we have
\[
\int_{\Omega} v_n\,dx \to 
\int_{\Omega} w\,dx=0 \qquad \mbox{as } n\to \infty. 
\]
Therefore, $c_n \to 0$ as $n \to \infty.$ Thus, the convergence $v_n \to w$ in $L^2(\Omega)$ and 
the fact that $u \in L^2(\Omega)$ together
imply that $w_n \to w$ in $L^2(\Omega).$
Thus, $w \in {\overline{{\cA}_u}}^{L^2(\Omega)}$, and $\cB_u \subseteq {\overline{\cA_u}}^{L^2(\Omega)}$.
It is clear that $\cB_u$ is closed in $L^2(\Omega)$. Thus, 
since $\cA_u \subseteq \cB_u,$ we have 
$\cB_u = {\overline{\cA_u}}^{L^2(\Omega)}.$
\end{proof}

For any $p \in [1, \infty)$ and any measurable $A \subset \R^3$ with the Lebesgue measure $|A| > 0$, we denote 
\[
L_0^p(A) = \left\{ w \in L^p(A): \int_A w \, dx = 0 \right\}. 
\]
The following elementary result will be used later: 

\begin{lemma}
\label{l:ubar}
Let $A$ be a measurable subset of $\R^3$ with $0<|A|<\infty$ and $u\in L^2(A)$. Suppose
\begin{equation}
\label{Auv}
\int_{A} u v \, dx=0 \qquad \forall v \in L_0^2(A). 
\end{equation}
Then $u$ is constant a.e.\ in $A$. 
\end{lemma}

\begin{proof}
Denoting 
\[
\overline{u}= \frac{1}{|A|}\int_{A} u \, dx, 
\]
we have $u - \overline{u} \in L_0^2(A)$. It thus follows from \eqref{Auv} with $v = u - \overline{u}$ that
\[
\int_A (u-\overline{u})^2 dx = \int_A u (u - \overline{u})\, dx - \overline{u}\int_A (u - \overline{u})\, dx =0. 
\]
Hence, $u = \overline{u}$ a.e.\ in $A.$
\end{proof}

The following lemma provides some necessary conditions for a critical point of
the functional $J: \cE \to \R:$ 

\begin{lemma}
    \label{l:EL4J} Let $u\in \cE$ be a critical point of $J: \cE \to \R$.  
    Let $\rho_u = \cT_g(u) \in H_g^1(\Omega)\cap H^2(\Omega)$ and
    $\psi_u = \cL_u (uf'(\rho_u)) \in H_0^1(\Omega)\cap H^2(\Omega)$
be given as in Theorem~\ref{t:rhou} and Lemma~\ref{l:Lu}, respectively. 
    Then, the following hold true: 
\begin{compactenum}
    \item[{\rm (1)}]
    There exists  a constant $c_u \ge 0$ such that 
\begin{equation}
\label{cu}
f(\rho_u) ( 1 - \psi_u)=c_u \quad \mbox{ a.e.\ in }\,  \{u>0\}.
\end{equation}
\item[{\rm (2)}] The measure $| \{ u  = 0 \} | > 0.$
\end{compactenum}
\end{lemma}

\begin{proof}
(1) Note that $| \{ u > 0 \} | > 0$ as  $u \in \cE.$
By Definition~\ref{d:CriticalPoint} and Theorem \ref{t:FirstVariations}, we have
\[
J'[u][v] = \int_\Omega f(\rho_u) ( 1 - \psi_u) v \, dx=0 \qquad \forall  v \in \cA_u
\]
Since $f(\rho_u) ( 1 - \psi_u)\in L^2(\Omega),$ it follows from Lemma~\ref{l:AuBu} that
\[
\int_\Omega f(\rho_u) ( 1 - \psi_u) v \, dx=0 \qquad \forall v \in \cB_u = \overline{\cA_u}^{L^2(\Omega)}.
\]
By the definition of $\cB_u$, this is equivalent to 
\[
\int_{\{ u > 0 \}}  f(\rho_u) ( 1 - \psi_u) v \, dx=0 \qquad \forall v \in L_0^2(\{ u > 0 \}).
\]
Now, it follows from Lemma~\ref{l:ubar} that \eqref{cu} holds true for some constant
$c_u \in \R.$ Since $\rho_u \ge 0$ in $\Omega$ by Theorem~\ref{t:rhou}
and $\psi_u \le 1$ in $\Omega$ by Lemma~\ref{l:Lu}, we have $c_u \ge 0.$

 (2) Suppose that $|\{ u =0\}|=0$, i.e., $u >0 $ a.e.\ in $\Omega$.
By Part {\rm (1)}, $f(\rho_u)(1-\psi_u)=c_u $ a.e.\ in 
$\Omega.$ Since $\rho_u,\psi_u\in H^2(\Omega)\hookrightarrow C(\overline{\Omega})$, 
$f(\rho_u)(1-\psi_u)= c_u$ pointwise in $\overline{\Omega}.$
On $\partial\Omega$, we have $\psi_u=0$ and $\rho_u=g$. 
Note that $y_0 \in \partial \Omega$ and $g(y_0) = g_{\rm max} > 0.$ 
Therefore,  $c_u = f(g(y_0)) = f(g_{\rm max}) > 0.$
Since $0 \le \psi_u \le 1 $ in $\Omega$ by Lemma~\ref{l:Lu} and
$0\le \rho_u <  g_{\rm max} $ in $\Omega$ by Theorem~\ref{t:rhou}, and 
since $f$ is strictly increasing,
we have 
\[
f(g_{\rm max}) = c_u=f(\rho_u) (1  - \psi_u) \le f(\rho_u) < f(g_{\rm max}) \qquad \mbox{in } \Omega.
\]
This is impossible. 
\end{proof}

We are now ready to prove the main theorem in this section. 

\begin{proof}[Proof of Theorem~\ref{t:Jmax}]
    Suppose $u \in \cE$ were a local maximizer of $J: \cE\to \R$. By Lemma~\ref{l:Jpu}, it
    is a critical point of $J: \cE \to \R.$
    Let $\rho_u = \cT_g(u)$ and 
    $\psi_u = \cL_u(uf'(\rho_u))$. By Lemma~\ref{l:EL4J}, there exists a constant $c_u \ge 0$ such that 
    \begin{equation}
    \label{cu1}
        f(\rho_u) (1 - \psi_u) = c_u \qquad \mbox{a.e.\ on } \{ u > 0 \}. 
    \end{equation}
We show that
\begin{equation}
    \label{cu2}
    f(\rho_u) (1 - \psi_u) \le c_u \qquad \mbox{a.e.\ on } \{ u = 0 \}. 
\end{equation}
Note by Lemma~\ref{l:EL4J} that $|\{ u = 0 \}| > 0.$ It suffices to show that
\begin{equation} 
\label{cuQ}
f(\rho_u) (1-\psi_u) \le c_u \qquad \mbox{a.e. on } Q
\end{equation}
for any measurable $Q \subseteq \{ u = 0 \}$ with $| Q | > 0$.

Since $u\in \cE$ is a local maximizer of $J: \cE \to \R$, it follows from Lemma~\ref{l:Jpu}
that $J'_+[u][v]\le 0$ for any $v \in \cA_u^+.$ Thus, we have by 
 Theorem ~\ref{t:FirstVariations} and \eqref{cu1} that 
\begin{align}
\label{Jpuv+}
    J'_+[u][v] & = \int_{\{ u > 0 \} } f(\rho_u) (1 - \psi_u) v \, dx + 
    \int_{\{ u = 0 \}} f(\rho_u) ( 1 - \psi_u) v\, dx \nonumber \\
    & = c_u \int_{\{ u > 0 \}} v \, dx  + \int_{\{ u = 0 \}} f(\rho_u) ( 1 - \psi_u) v\, dx 
    \nonumber \\
    & \le 0 \qquad \forall v\in \cA_u^+. 
\end{align}
We will construct a family of functions $w_\ve\in \cE$ for $\ve > 0$ small and
$v_\ve := w_\ve - u \in \cA_u^+$ so that we can apply this inequality to prove \eqref{cuQ}.
Figure~\ref{f:wve} illustrates the construction of $w_\ve.$

\begin{figure}[h]
\begin{center}
\includegraphics[width=4.4in,height=1.8in]{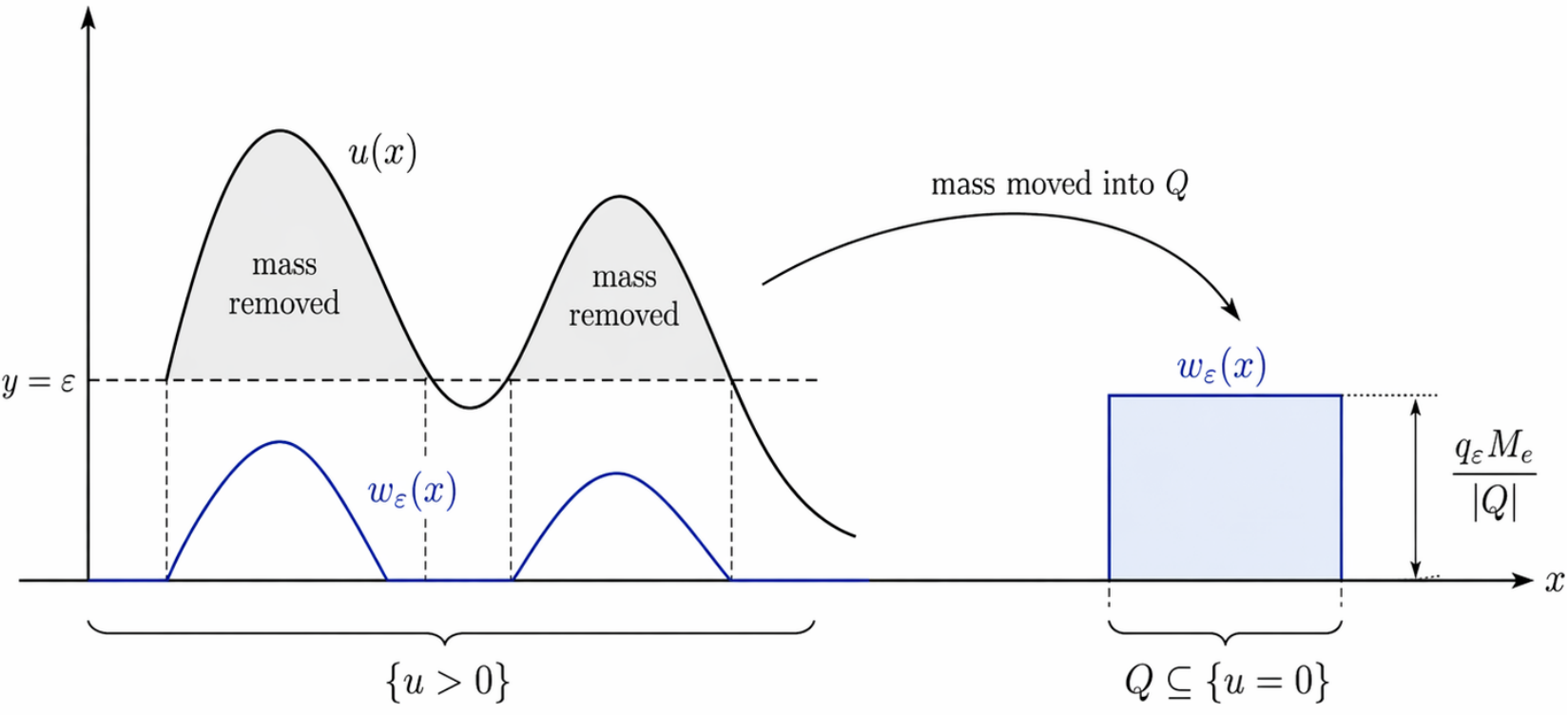}
\end{center}
\vspace{-4 mm}

\caption{Schematic of the construction of $w_\ve.$}
\label{f:wve}
\end{figure}

Since $ u \ge 0 $ a.e.\ in $\Omega$, 
$| \{ u > 0 \}| > 0$, and $\{ u > 0 \} = \cup_{n=1}^\infty \{ u > 1/n \}, $
there exists $\ve_0 \in (0, 1)$ such that 
 $|\{ u > \ve \}| > 0$ for any $\ve \in (0, \ve_0)$. Define 
\[
q_\ve = 1 - \frac{1}{M_{\rm e}} \int_{\{ u > \ve \}} (u - \ve) \, dx
\qquad \forall \ve \in (0, \ve_0).
\]
We have 
\[
q_\ve > \frac{1}{M_{\rm e}} \left( M_{\rm e} - \int_{\{ u > \ve \}} u  \, dx \right)
\ge\frac{1}{M_{\rm e}} \left( M_{\rm e} - \int_{\{ u  > 0 \}} u dx \right) = 0. 
\]
Moreover, by the Lebesgue Monotone Convergence Theorem, 
\[
\int_{\{ u > \ve \}} (u - \ve) \, dx 
= \int_{\{ u > 0 \}} u \chi_{\{ u > \ve \}}\, dx - \ve |\{ u > \ve \}| \to \int_{\{ u > 0 \}} u dx = M_{\rm e} 
\qquad \mbox{as } \ve \to 0^+.
\]
Consequently, $q_\ve \to 0$ as $\ve \to 0^+.$

Let $Q\subseteq \{ u = 0 \}$ be measurable and assume $| Q | > 0.$
Define for any $\ve \in (0, \ve_0)$
\[
w_{\ve} = \left\{
\begin{aligned}
    & \max\{ u - \ve, 0\} & \qquad \mbox{in } \{ u > 0 \}, \\
    & \frac{q_\ve M_{\rm e}}{|Q|} \chi_Q & \qquad \mbox{in } \{ u = 0 \}.
\end{aligned}
\right.
\]
It is clear that $w_\ve \in L^2(\Omega)$ and $w_\ve \ge 0$ in $\Omega.$ Moreover, 
\[
\int_\Omega w_\ve \, dx = \int_{\{ u > \ve\}}(u-\ve)\, dx 
+ q_\ve M_{\rm e} = (1-q_\ve) M_{\rm e} + q_\ve M_{\rm e} = M_{\rm e}.
\]
Thus, $w_\ve \in \cE.$ 
Define $v_\ve = w_\ve - u$ for any $\ve \in (0, \ve_0)$. Then, $u + t v_\ve = (1-t) u + t w_\ve \in \cE$ if $0 \le t \le 1.$ Hence, 
$v_\ve \in \cA_u^+.$

It follows from \eqref{Jpuv+} with $v$ replaced by $v_\ve$ and from the definition of $v_\ve$ and $w_\ve$ that 
\begin{align}
\label{Jpuve+}
        J'_+[u][v_\ve] &= c_u \int_{\{ u > 0 \}} v_\ve \, dx  + \int_{\{ u = 0 \}} f(\rho_u) ( 1 - \psi_u) w_\ve \, dx 
    \le 0.
\end{align}
By the definition of $w_\ve$, $v_\ve$, and $q_\ve$, we have 
\begin{align*}
c_u \int_{\{ u > 0 \}} v_\ve\, dx &=
c_u \int_{\{ u > 0 \} } (w_\ve - u) \, dx \\
&
=c_u\int_{\{ u > 0 \}} \left( \max\{ u - \ve, 0 \}-u\right)\, dx 
\\
& = c_u \int_{\{ u > \ve \}} (u-\ve)\, dx - c_u\int_{\{u > 0 \}}u\, dx \\
& = c_u M_{\rm e} ( 1 - q_\ve) - c_u M_{\rm e}\\
& = - c_u M_{\rm e} q_\ve. 
\end{align*}
On the other hand, by the definition of $w_\ve$, we have
\[
\int_{\{ u = 0 \}} f(\rho_u) (1 - \psi_u) w_\ve \, dx 
= \frac{q_\ve M_{\rm e}}{|Q|} \int_Q f(\rho_u) (1 - \psi_u)\, dx. 
\]
It follows now from \eqref{Jpuve+} that
\[
- c_u M_{\rm e} q_\ve + \frac{q_\ve M_{\rm e}}{|Q|} \int_Q f(\rho_u) (1 - \psi_u)\, dx\le 0
\qquad \forall \ve \in (0, \ve_0).
\]
Since $q_\ve > 0$ if $\ve \in (0, \ve_0)$, this implies that 
\[
\int_Q \left[ f(\rho_u) ( 1 - \psi_u) - c_u \right] \, dx \le 0. 
\]
Since this is true for all $Q\subseteq \{ u = 0 \}$ with $|Q| > 0$, we obtain \eqref{cu2}. 

 Since both $\rho_u$ and $\psi_u$ are continuous on $\overline{\Omega}$, 
 it now follows from \eqref{cu1} and \eqref{cu2} that 
\[
f(\rho_u) ( 1 - \psi_u) \le c_u \qquad \forall x \in \overline{\Omega}. 
\]
Recall that $y_0\in \partial \Omega$ satisfies 
$g(y_0) = g_{\rm max} := \max_{y \in \partial \Omega} g(y) > 0$. 
Since $\psi_u = 0 $ on $\partial \Omega$, we get by setting $x = y_0$ that 
$f(g_{\rm max} ) \le c_u. $ Consequently,  
since $\rho_u < g_{\rm max}$ in $\Omega$ by Theorem~\ref{t:rhou} and
$\psi_u \ge 0$ in $\Omega$ by Lemma~\ref{l:Lu}, we get
\[
f(g_{\rm max}) \le c_u = f(\rho_u) (1 - \psi_u) \le f(\rho_u) <f(g_{\rm max}) \qquad \mbox{a.e. in } \{ u > 0 \}. 
\]
This is a contradiction.
\end{proof}

\begin{remark}
    The proof above does not extend to proving possibly the non-existence of a critical point as the 
    use of the inequality $J'_+[u][v]\le 0$ for any $v \in \cA_u^+$ is crucial; cf.\ \eqref{Jpuv+}. 
    This inequality holds true for a local maximizer but not necessarily for a critical point.
\end{remark}

\section{Existence of Maximizers of a Regularized Reaction Functional}
\label{s:Regularization}

We recall that the regularized functional $\hat{J}$ is defined in \eqref{Jveu},
\[
\hat{J} [u]= \int_\Omega \left[ u f(\rho_u) - \frac{1}{\beta} u \log u - \frac{\ve}{2} |\nabla u|^2 \right] dx,
\]
We denote $\hat{\cE} = \cE \cap H^1(\Omega).$
\begin{theorem}
The functional $\hat{J} : \hat{\cE} \to \R$ admits a maximizer. 
\end{theorem}
\begin{proof}
Note that $s\mapsto s \log s $ is minimized at $s = e^{-1}$ with the minimum value $-e^{-1}.$ 
Thus, by Theorem~\ref{t:supJ}, 
we have 
\[
\hat{J}[u ] \le M_{\rm e} f(g_{\rm max}) + (\beta e)^{-1} |\Omega|\qquad \forall 
u \in \hat{\cE}. 
\]
Now, the constant function ${\bar{u}} :=M_{\rm e}/{|\Omega|} \in \hat{\cE}$ and 
\[
\hat{J}[\bar{u}]=J[\bar{u}]-M_{\rm e}{\beta}^{-1} \log \left( M_{\rm e}|\Omega|^{-1}\right) > -\infty.
\]
Therefore, ${-\infty<\alpha:=\sup_{u\in \hat{\cE}} \hat{J}[u] < +\infty.}$

Let $\{ u_n\} \subset \hat{\cE}$ be such that $ \hat{J}[u_n] \ge \alpha - 1 $ for all $n \ge 1$ 
and $\hat{J}[u_n] \to \alpha$ as 
$n\to \infty.$
Since each $u_n\in \hat{\cE}$, the bound on $J$ yields
\[
\alpha-1\le \hat{J} [u_n] 
\le M_{\rm e} f(g_{\rm max})+\frac{|\Omega|}{\beta e} - \frac{\ve}{2}\int_\Omega |\nabla u_n|^2\,dx
\]
Hence,
\[
\frac{\ve}{2}\int_\Omega |\nabla u_n|^2\,dx\le M_{\rm e} f(g_{\rm max}) + \frac{|\Omega |}{\beta e}-\alpha+1. 
\]
Thus, $\{\nabla u_n\}$ is bounded in $L^2(\Omega)$. Since the integral of 
each $u_n$ over $\Omega$ is $M_{\rm e}$, 
by the Poincar\'e inequality (cf.\ Theorem~1 in 5.8.1 of \cite{EvansBook2010}), 
$\{u_n\}$ is bounded in $H^1(\Omega)$. By the compact embedding
$H^1(\Omega)\hookrightarrow L^2(\Omega)$, there exist $u\in H^1(\Omega)$ and a subsequence
of $\{ u_n \}$,
not relabeled, such that
$u_n\rightharpoonup u $ in $H^1(\Omega),$
$u_n\to u $ in $L^2(\Omega)$, and $u_n(x)\to u(x)$ for a.e.\ $x\in \Omega.$
It is clear that $u \in \hat{\cE}.$

By Lemma~\ref{l:rhouk}, we have $J[u_n]\to J[u].$
On the other hand, by the weak lower semi-continuity, we have
\[
\liminf_{n\to\infty}\int_\Omega |\nabla u_n|^2\,dx \ge \int_\Omega |\nabla u|^2\,dx.
\]
By Fatou's lemma, 
\[
\liminf_{n\to \infty } \int_\Omega u_n \log u_n \, dx \ge \int_\Omega u \log u\, dx.
\]
Therefore,
$\alpha = \limsup_{n\to \infty} \hat{J}[u_n] \le \hat{J}[u] \le \alpha. $
Thus, $u$ is a maximizer of $\hat{J}$ over $\hat{\cE}$.
\end{proof}

\section*{Acknowledgment}

This work was supported in part by the US National Science Foundation
through the grant DMS-2208465. B.L.\ thanks Professor Ulrich Gerland 
for helpful discussions.

\setlength{\bibsep}{0.0pt}
\bibliography{Localization_References}
\bibliographystyle{plain}

\end{document}